\documentclass[a4paper]{amsart}
\usepackage{graphicx}
\usepackage{amssymb}
\usepackage{amsmath}
\usepackage{amsthm}
\usepackage[all]{xy}
\usepackage{ragged2e}
\usepackage{makecell}
\usepackage{enumitem}

\newtheorem*{thmmA}{\bf Theorem A}
\newtheorem*{thmmB}{\bf Theorem B}
\newtheorem{thm}{Theorem}[section]

\newtheorem{lem}[thm]{Lemma}
\newtheorem{exm}[thm]{Example}
\newtheorem{prop}[thm]{Proposition}
\theoremstyle{definition}
\newtheorem{defn}[thm]{Definition}
\theoremstyle{remark}
\newtheorem{rem}[thm]{\bf Remark}

\newcommand{\xr}{\mathbb{R}}
\newcommand{\xh}{\mathbb{H}}
\newcommand{\xc}{\mathbb{C}}
\newcommand{\dba}{\mathbf{D}^{\mathrm{b}}(A\mbox{-{\rm mod}})}
\begin{document}
\bibliographystyle{plain}
\title{the derived-discrete algebras over the real numbers}

\author{Jie Li}

\thanks{E-mail: lijie0$\symbol{64}$mail.ustc.edu.cn}

\begin{abstract}
We classify derived-discrete algebras over the real numbers up to Morita equivalence, using the classification of complex derived-discrete algebras in [{\sc D. Vossieck}, {\em The algebras with discrete derived category}, J. Algebra {\bf 243} (2001), 168--176]. To this end, we investigate the quiver presentation of the complexified algebra of a real algebra given by a modulated quiver and an admissible ideal.	
\end{abstract}

\maketitle

\section{introduction}

The notion of derived-discrete algebras was introduced by Vossieck in \cite{V}. Their bounded derived categories are virtually of finite representation type. More precisely, for each cohomology dimension vector, there are only finitely many objects in the bounded derived category. This class of algebras is important in the representation theory of finite dimensional algebras; see \cite{BGS, BPP}. 

The main result in \cite{V} states that a connected derived-discrete algebra over an algebraically closed field either is piecewise hereditary of Dynkin type, or admits a quiver presentation which is gentle one-cycle without clock condition. Our motivation is to give a similar description of derived-discrete algebras over a \textit{non-algebraically closed} field. 

One possible approach is to generalize the arguments in \cite{V} to non-algebraically closed field versions. For example, the covering theory of quivers plays an important role in \cite{V}. However, over a non-algebraically closed field, it seems very complicated to develop such a theory for modulated quivers. 

On the other hand, derived-discrete algebras and piecewise hereditary algebras are compatible with finite separable field extensions; see \cite{L}. It suggests that field extensions might be useful for our purpose. Assume that $k$ is a perfect field which is cofinite in its algebraic closure $\overline{k}$. To describe a derived-discrete algebra over $k$, we need to consider its presentation $T(Q,\mathcal{M})/I$, where $Q$ is a quiver, $\mathcal{M}$ a modulation on $Q$ and $I$ an admissible ideal of the tensor algebra $T(Q,\mathcal{M})$; see \cite{B,DR}. To make use of the classification of derived-discrete algebras over $\overline{k}$, we also need to describe the quiver presentation $\overline{k}\Gamma/J$ of $T(Q,\mathcal{M})/I\otimes_k\overline{k}$, the field extension of $T(Q,\mathcal{M})/I$. However, for a general $k$, the modulation $\mathcal{M}$ is complicated, which makes it difficult to describe both $T(Q,\mathcal{M})/I$ and $\overline{k}\Gamma/J$.

In this article, we take $k=\xr$ and $\overline{k}=\xc$, in which case, $\mathcal{M}$ and $\Gamma$ can be explicitly described. Indeed, Gabriel described $\Gamma$ in \cite{G} and used it to classify all the representation-finite hereditary $\xr$-algebras. Since we need to describe the ideals $I$ and $J$, we should go further. We give the related isomorphisms explicitly and do some concrete computation; see Lemma~\ref{sym} for example. Such computation will become too difficult if $k$ is a general field.

Our main theorem is the following.

\begin{thmmA}{\rm (Theorem \ref{main})}
	Let $\xc\Gamma/J$ be a quiver presentation of $T(Q, \mathcal{M})/I\otimes_{\xr}\xc$. Then $T(Q,\mathcal{M})/I$ is gentle one-cycle without clock condition if and only if each connected component of $\xc\Gamma/J$ is gentle one-cycle without clock condition.
\end{thmmA}

Here, since an $\xr$-algebra $T(Q,\mathcal{M})$ is in general not a path algebra, the definition of \textit{gentle one-cycle without clock condition} for $T(Q, \mathcal{M})/I$ is not the usual one in \cite{V}. It requires that the modulation on the vertices of $Q$ is uniform, which makes $T(Q,\mathcal{M})$ 'similar' to path algebra (see Example \ref{notpathalg} for a nontrivial example). Then the ideal $I$ can be generated by elements in the modulations on paths and we can define 'gentle one-cycle without clock condition'; see Definition~4.2.

For the proof of Theorem A. If the modulation on the vertices is uniform, then $J$ can be described; see Lemma~\ref{rtc}. This helps us to prove the `only if' part of the theorem. For the `if' part, it needs more effort. One problem is that $J$ is not unique. We solve this by investigating when $J$ is monomial; see Proposition~\ref{special}. Another problem is why the modulation is uniform on vertices. We prove this by giving some restrictions on $Q$ and $\mathcal{M}$, which are caused by the combinatorial conditions of $\xc\Gamma/J$; see lemmas in Section~6.

In view of the classification of complex derived-discrete algebras in \cite{V} and \cite[Theorem 4.1]{L}, we can classify real derived-discrete algebras by Theorem A:

\begin{thmmB}{\rm (Theorem 4.6)}
	A connected $\xr$-algebra is derived-discrete if and only if it is either piecewise hereditary of Dynkin type or Morita equivalent to $T(Q,\mathcal{M})/I$ which is gentle one-cycle without clock condition.
\end{thmmB}

The article is organized as follows. In Section~2, we list division rings over $\xr$, simple bimodules over these rings and their complexifications. Given a modulated quiver $(Q,\mathcal{M})$, we construct in Section~3 a quiver presentation $\xc\Gamma/J$ of $T(Q,\mathcal{M})/I\otimes_{\xr}\xc$; see Theorem~\ref{morita eq}. In Section 4, we state Theorem~A and prove Theorem~B; see Theorem~\ref{main} and Theorem~\ref{cls}. The proofs of the `only if' part and `if' part of Theorem~\ref{main} are put in Section~5 and Section~6, respectively.

\section{Real algebras and complexifications}
In this section, we give some well known isomorphisms on $\xr$-algebras and their complexifications. It is hard to see these materials in one literature. 

By Frobenius theorem \cite[Theorem 4.6.1]{DK}, the (isomorphism classes of) finite-dimensional division rings over the real number field are only the field $\mathbb{R}$ itself, the field $\mathbb{C}$ of complex numbers and the quaternion algebra $\mathbb{H}$. Denote the conjugate of $c\in\xc$ by $\overline{c}$. We view $\mathbb{H}$ as $$\{\begin{bmatrix} a & b \\ -\overline{b} & \overline{a} \end{bmatrix}
|a,b\in \mathbb{C} \},$$ an $\mathbb{R}$-subalgebra of $M_2(\mathbb{C})$ (the full matrix algebra of the $2\times2$ matrices over $\mathbb{C}$). In this section, let $a$, $b$, $c$, $d$ be any complex numbers.

The complexifications of the above division rings are listed below.

\subsection{The division ring $\xr$} We have an isomorphism of $\xc$-algebras $$\xr\otimes_{\xr}\xc\simeq\xc \mbox{ given by multiplication.}$$
We denote such $\xc$, the complexification of $\xr$, by $C_{\xr}$.
\subsection{The division ring $\xh$} We have an isomorphism of $\xc$-algebras (see \cite[Theorem 4.5.1]{DK}) $$\xh\otimes_{\xr}\xc\simeq M_2(\xc)\mbox{ given by }\begin{bmatrix} a & b \\ -\overline{b} & \overline{a} \end{bmatrix}\otimes c\mapsto\begin{bmatrix} ac & bc\\-\overline{b}c &\overline{a}c\end{bmatrix}.$$
We denote such algebra $M_2(\xc)$ by $C_{\xh}$. In $C_{\xh}$, $\begin{bmatrix}1&0\\0&0\end{bmatrix}$ is an idempotent.

\subsection{The division ring $\xc$} We have an isomorphism of $\xc$-algebras $$\xc\otimes_{\xr}\xc\simeq\xc\times\xc\mbox{ given by } a\otimes b\mapsto (ab,\overline{a}b).$$ We denote such algebra $\xc\times\xc$ by $C_{\xc}$. In $C_{\xc}$, $(1,0)$ and $(0,1)$ are two idempotents.

\medskip

Now we consider the complexifications of bimodules. Let $D_1$, $D_2$ be two real division rings and $M$ be a $D_1$-$D_2$-bimodule, which is also a $D_1\otimes D_2^{\mathrm{op}}$ left module. Then $M\otimes_{\xr}\xc$ is a $D_1\otimes_{\xr}\xc$-$D_2\otimes_{\xr}\xc$-bimodule. Up to isomorphisms, all the simple bimodules over real simple division rings and their complexifications are listed below.

\subsection{Simple $\xr$-$\xr$-bimodules} Since $\xr\otimes_{\xr}\xr^{\mathrm{op}}\simeq\xr$, $_{\xr}\xr_{\xr}$ is the only simple $\xr$-$\xr$-bimodule. Via the isomorphism in Subsection 2.1,  $$_{\xr}\xr_{\xr}\otimes_{\xr}\xc\simeq{\xc}\mbox{, given by multiplication,}$$ is an isomorphism of $C_\xr$-$C_\xr$-bimodules, where the module structure of ${\xc}$ is given by multiplication.
Denote such bimodule $\xc$ by $_{\xr}C_{\xr}$.

\subsection{Simple $\xr$-$\xc$-bimodules} Since $\xr\otimes_{\xr}\xc^{\mathrm{op}}\simeq\xc$, $_{\xr}\xc_{\xc}$ is the only simple $\xr$-$\xc$-bimodule. Via isomorphisms in Subsections 2.1 and 2.3,  $$_{\xr}\xc_{\xc}\otimes_{\xr}\xc\simeq{\xc\oplus\xc}\mbox{, given by }(a\otimes b)\mapsto(ab,\overline{a}b),$$ is an isomorphism of $C_\xr$-$C_\xc$-bimodules, where the module structure of $\xc\oplus\xc$ is given by $c\curvearrowright(a,b)=(ca,cb)$ and $(a,b)\curvearrowleft(c,d)=(ac,bd)$. We denote such bimodule $\xc\oplus\xc$ by ${_{\xr}C_{\xc}}$.

\subsection{Simple $\xc$-$\xr$-bimodules} Since $\xc\otimes_{\xr}\xr^{\mathrm{op}}\simeq\xc$, ${_{\xc}\xc_{\xr}}$ is the only simple $\xc$-$\xr$-bimodule. Via isomorphisms in Subsections 2.1 and 2.3,  $$_{\xc}\xc_{\xr}\otimes_{\xr}\xc\simeq{\xc\oplus\xc}\mbox{, given by }(a\otimes b)\mapsto(ab,\overline{a}b),$$ is an isomorphism of $C_\xc$-$C_\xr$-bimodules, where the module structure of $\xc\oplus\xc$ is given by $(c,d)\curvearrowright(a,b)=(ca,db)$ and $(a,b)\curvearrowleft c=(ac,bc)$. We denote such bimodule $\xc\oplus\xc$ by ${_{\xc}C_{\xr}}$.

\subsection{Simple $\xr$-$\xh$-bimodules} Since $\xr\otimes_{\xr}\xh^{\mathrm{op}}\simeq\xh^{\mathrm{op}}$, $_{\xr}\xh_{\xh}$ is the only simple $\xr$-$\xh$-bimodule. Via isomorphisms in Subsections 2.1 and 2.2,  $$_{\xr}\xh_{\xh}\otimes_{\xr}\xc\simeq M_2(\xc) \mbox{, given by }\begin{bmatrix} a & b \\ -\overline{b} & \overline{a} \end{bmatrix}\otimes c\mapsto\begin{bmatrix} ac & bc\\-\overline{b}c &\overline{a}c\end{bmatrix},$$ is an isomorphism of $C_\xr$-$C_\xh$-bimodules, where the module structure of $M_2(\xc)$ is given by scalar multiplication from the left and matrix multiplication from the right. We denote such bimodule $M_2(\xc)$ by ${_{\xr}C_{\xh}}$.

\subsection{Simple $\xh$-$\xr$-bimodules} Since $\xh\otimes_{\xr}\xr^{\mathrm{op}}\simeq\xh$, $_{\xh}\xh_{\xr}$ is the only simple $\xh$-$\xr$-bimodule. Via isomorphisms in Subsections 2.1 and 2.2, $$_{\xh}\xh_{\xr}\otimes_{\xr}\xc\simeq M_2(\xc)\mbox{, given by }\begin{bmatrix} a & b \\ -\overline{b} & \overline{a} \end{bmatrix}\otimes c\mapsto\begin{bmatrix} ac & bc\\-\overline{b}c &\overline{a}c\end{bmatrix},$$ is an isomorphism of $C_\xh$-$C_\xr$-bimodules, where the module structure of $M_2(\xc)$ is given by matrix multiplication from the left and scalar multiplication from the right. We denote such bimodule $M_2(\xc)$ by ${_{\xh}C_{\xr}}$.

\subsection{Simple $\xh$-$\xh$-bimodules} Since $\xh\otimes_{\xr}\xh^{\mathrm{op}}\simeq M_4(\xr)$ (see \cite[Theorem 4.3.1]{DK}), $_{\xh}\xh_{\xh}$ is the only simple $\xh$-$\xh$-bimodule. Via isomorphism in Subsection 2.2,  $$_{\xh}\xh_{\xh}\otimes_{\xr}\xc\simeq M_2(\xc)\mbox{, given by }\begin{bmatrix} a & b \\ -\overline{b} & \overline{a} \end{bmatrix}\otimes c\mapsto\begin{bmatrix} ac & bc\\-\overline{b}c &\overline{a}c\end{bmatrix},$$ is an isomorphism of $C_\xh$-$C_\xh$-bimodules, where the module structure of $M_2(\xc)$ is given by matrix multiplication. We denote such bimodule $M_2(\xc)$ by ${_{\xh}C_{\xh}}$.

\subsection{Simple $\xh$-$\xc$-bimodules} Since $\xh\otimes_{\xr}\xc^{\mathrm{op}}\simeq M_2(\xc)$ as algebras, $_{\xh}\xc^2_{\xc}$ (2-dimensional column vectors over $\xc$) is the only simple $\xh$-$\xc$-bimodule. Via isomorphisms in Subsections 2.2 and 2.3, we can check that $$_{\xh}\xc^2_{\xc}\otimes_{\xr}\xc\simeq M_2(\xc)\mbox{, given by }\begin{bmatrix} a\\b \end{bmatrix}\otimes c\mapsto\begin{bmatrix} ac&-\overline{b}c\\bc&\overline{a}c \end{bmatrix},$$ is an isomorphism of $C_\xh$-$C_\xc$-bimodules, where the module structure of $M_2(\xc)$ is given by matrix multiplication after viewing $(a,b)\in C_\xc$ as $\begin{bmatrix}a&0\\0&b\end{bmatrix}\in M_2(\xc)$. We denote such bimodule $M_2(\xc)$ by ${_{\xh}C_{\xc}}$.

\subsection{Simple $\xc$-$\xh$-bimodules} Since $\xc\otimes_{\xr}\xh^{\mathrm{op}}\simeq M_2(\xc)^{\mathrm{op}}$ as algebras, ${_{\xc}\xc^2_{\xh}}$ (2-dimensional row vectors over $\xc$) is the only simple $\xc$-$\xh$-bimodule. Via isomorphisms in Subsections 2.2 and 2.3, we can check that  $${_{\xc}\xc^2_{\xh}}\otimes_{\xr}\xc\simeq M_2(\xc)\mbox{, given by }\begin{bmatrix} a&b \end{bmatrix}\otimes c\mapsto\begin{bmatrix} ac&bc\\-\overline{b}c&\overline{a}c \end{bmatrix},$$ is an isomorphism of $C_\xc$-$C_\xh$-bimodules, where the module structure of $M_2(\xc)$ is given by matrix multiplication after viewing $(a,b)\in C_\xc$ as $\begin{bmatrix}a&0\\0&b\end{bmatrix}\in M_2(\xc)$. We denote such bimodule $M_2(\xc)$ by ${_{\xc}C_{\xh}}$.

\subsection{Simple $\xc$-$\xc$-bimodules} Since $\xc\otimes_{\xr}\xc\simeq\xc\times\xc$ as algebras, ${_{\xc}\xc_{\xc}}$ and $_{\xc}\overline{\xc}_{\xc}$ are all the simple $\xc$-$\xc$-bimodules, where ${_{\xc}\overline{\xc}_{\xc}}$ equals $\xc$ as a set with the module structure given by $a\curvearrowright c \curvearrowleft b=ac\overline{b}$. Via isomorphism in Subsection~2.3,  $${_{\xc}\xc_{\xc}}\otimes_{\xr}\xc\simeq\xc\oplus\xc\mbox{, given by }a\otimes c\mapsto (ac,\overline{a}c),$$ $${_{\xc}\overline{\xc}_{\xc}}\otimes_{\xr}\xc\simeq\xc\oplus\xc\mbox{, given by }a\otimes c\mapsto (ac,\overline{a}c),$$ are isomorphisms of $C_\xc$-$C_\xc$-bimodules, where the module structure of the first $\xc\oplus\xc$ is given by $(a,b)\curvearrowright(c,d)=(ac,bd)$ and $(c,d)\curvearrowleft(a,b)=(ca,db)$, and the module structure of the second one is given by $(a,b)\curvearrowright(c,d)=(ac,bd)$ and $(c,d)\curvearrowleft(a,b)=(cb,da)$. We denote the first bimodule $\xc\oplus\xc$ by ${_{\xc}C_{\xc}}$ and the second one by ${_{\xc}\overline{C}_{\xc}}$.

\begin{rem}
	The isomorphisms given in this section are not unique. We fix them since they will be used in the next section (related to the construction of $\Gamma$ and the isomorphism $\Psi$). Under some slight modifications, it makes no essential differences if we choose other isomorphisms.
\end{rem}

\section{The complexified quiver presentation}

\subsection{Tensor algebras of modulated quivers over $\xr$}
Let $Q=(Q_0,Q_1,s,t)$ be a finite quiver, where $Q_0$ is a finite set of vertices, $Q_1$ is a finite set of arrows, $s\colon Q_1\rightarrow Q_0$ maps an arrow to its starting vertex and $t\colon Q_1\rightarrow Q_0$ maps an arrow to its terminal vertex.

A \textbf{modulation} $\mathcal{M}$ of $Q$ (over $\xr$) is a map given below: on each vertex $i$, $\mathcal{M}(i)$ is a division ring $\xr$, $\xh$ or $\xc$; on each arrow $\alpha$, $\mathcal{M}(\alpha)$ is one of simple $\mathcal{M}(t(\alpha))-\mathcal{M}(s(\alpha))$-bimodules in Section~2. In this case, we call $(Q,\mathcal{M})$ a \textbf{modulated quiver} (over $\xr$); compare definitions in \cite{B,DR}.

The tensor algebra of a modulated quiver $(Q,\mathcal{M})$ is the tensor algebra of the bimodule $\underset{\alpha\in Q_1}{\oplus}\mathcal{M}(\alpha)$ over the ring $\underset{i\in Q_0}{\prod}\mathcal{M}(i)$, that is, $$T(Q,\mathcal{M}):=(\underset{i\in Q_0}{\prod}\mathcal{M}(i))\oplus(\underset{\alpha\in Q_1}{\bigoplus}\mathcal{M}(\alpha))\oplus(\underset{\alpha\in Q_1}{\bigoplus}\mathcal{M}(\alpha))^{\otimes2}\oplus\cdots.$$

For a path $p=\alpha_n\alpha_{n-1}\cdots\alpha_1$ in $Q$, set $$\mathcal{M}(p):=\mathcal{M}(\alpha_n)\otimes_{\mathcal{M}(s(\alpha_n))}\mathcal{M}(\alpha_{n-1})\otimes_{\mathcal{M}(s(\alpha_{n-1}))}\cdots\otimes_{\mathcal{M}(s(\alpha_2))}\mathcal{M}(\alpha_1).$$ In this article, when we say path, we mean a path of positive length. Since the bimodule structure is compatible with the composition of paths, $$T(Q,\mathcal{M})=(\underset{i\in Q_0}{\prod}\mathcal{M}(i))\oplus(\underset{p \mbox{ path in }Q}{\bigoplus}\mathcal{M}(p)).$$

\begin{defn}
	We say that the modulation $\mathcal{M}$ in a modulated quiver $(Q,\mathcal{M})$ over $\xr$ is \textbf{v-uniform} (with $D$) if $\mathcal{M}(i)=\mathcal{M}(j)(=D$, where $D$ is $\xr$, $\xh$ or $\xc)$ for any $i,j\in Q_0$.
\end{defn}

If $\mathcal{M}$ is v-uniform with $D$, then for each path $p=\alpha_n\alpha_{n-1}\cdots\alpha_1$ in $Q$,  $$\mathcal{M}(p):=\mathcal{M}(\alpha_n)\otimes_{D}\mathcal{M}(\alpha_{n-1})\otimes_{D}\cdots\otimes_{D}\mathcal{M}(\alpha_1)$$ is a simple $D$-$D$-bimodule generated by any nonzero element. Since the simple $D$-$D$-bimodule $\mathcal{M}(\alpha)$ equals $D$ as a set, it contains an element  $1_{\mathcal{M}(\alpha)}=1_D$. Also, set $$1_{\mathcal{M}(p)}:=1_{\mathcal{M}(\alpha_n)}\otimes 1_{\mathcal{M}(\alpha_{n-1})}\otimes\cdots\otimes 1_{\mathcal{M}(\alpha_1)}\in\mathcal{M}(p).$$ Then each element of $T(Q,\mathcal{M})$ can be written as a $D$-linear combination of the basis $\{1_{\mathcal{M}(p)}\;|\;p \mbox{ path in }Q\}$.

Recall that the path algebra of $Q$ over the division ring $D$ is defined as $$DQ:=(\underset{i\in Q_0}{\prod}De_i)\oplus(\underset{p \mbox{ path in }Q}{\bigoplus}Dp),$$ where $e_i$ is the identity ``path'' for each vertex $i$, and multiplication is induced by composition of paths.
\begin{lem}\label{pathalg}
	We have the following statements for a modulated quiver $(Q,\mathcal{M})$.\\
	1) If $\mathcal{M}$ is v-uniform with $\xr$, then $T(Q,\mathcal{M})\simeq \xr Q$.\\
	2) If $\mathcal{M}$ is v-uniform with $\xh$, then $T(Q,\mathcal{M})\simeq \xh Q$.\\
	3) If $\mathcal{M}$ is v-uniform with $\xc$ and for each $\alpha\in Q_1$, $\mathcal{M}(\alpha)=\xc$, then $T(Q,\mathcal{M})\simeq \xc Q$.
\end{lem}
\begin{proof}
	Assume that $\mathcal{M}$ is v-uniform with $D$. For each $i\in Q_0$, there is an isomorphism of algebras $f_0\colon\mathcal{M}(i)\rightarrow De_i$, mapping $1_{\mathcal{M}(i)}$ to $1e_i$. For each $\alpha\in Q_1$, via $f_0$, there is an isomorphism of bimodules $f_1\colon\mathcal{M}(\alpha)\rightarrow D\alpha$, mapping $1_{\mathcal{M}(\alpha)}$ to $1\alpha$. Then $f_0$ and $f_1$ induce an isomorphism of algebras $T(Q,\mathcal{M})\simeq DQ$.
\end{proof}

When the modulation of $Q$ is v-uniform with $\xc$, $T(Q,\mathcal{M})$ is not always isomorphic to a path algebra; see Example \ref{notpathalg}. In this case, sometimes we can change the modulation as below. 

Given a vertex $v$ in $Q$, set $v^+:=\{\alpha\in Q_1\,|\,s(\alpha)=v\},\mbox{ }v^-:=\{\alpha\in Q_1\,|\,t(\alpha)=v\}.$ Assume that there are no loops on $v$ (i.e. $v^+\cap v^-=\emptyset$). A new modulation $\mathcal{M'}$ of $Q$ is defined to be: $$\forall i\in Q_0\mbox{, }\mathcal{M'}(i)=\xc,$$
$$\forall\alpha\in Q_1\mbox{, }\mathcal{M'}(\alpha)=\begin{cases}
\xc & \text{if } \alpha\in v^+\cup v^-,\mathcal{M}(\alpha)=\overline{\xc}\\
\overline{\xc} & \text{if } \alpha\in v^+\cup v^-,\mathcal{M}(\alpha)=\xc\\
\mathcal{M}(\alpha) &\text{if } \alpha\notin v^+\cup v^-
\end{cases}.$$

Now let $$\phi_0\colon\underset{i\in Q_0}{\prod}\mathcal{M}(i)\rightarrow\underset{i\in Q_0}{\prod}\mathcal{M}'(i), (\lambda_i)_{i\in Q_0}\mapsto(\lambda'_i)_{i\in Q_0},$$ where $\lambda_i\in\xc$, 
$\lambda'_i=\begin{cases}
\lambda_i& \text{if } i\neq v\\\overline{\lambda_i} & \text{if } i=v
\end{cases}.$ Then $\phi_0$ is an isomorphism of $\xr$-algebras.

Let $$\phi_1\colon\underset{\alpha\in Q_1}{\oplus}\mathcal{M}(\alpha)\rightarrow\underset{\alpha\in Q_1}{\oplus}\mathcal{M}'(\alpha), (z_{\alpha})_{\alpha\in Q_1}\mapsto(z'_{\alpha})_{\alpha\in Q_1},$$ where $z_{\alpha}\in\xc$, 
$z'_{\alpha}=\begin{cases}
z_{\alpha}& \text{if } \alpha\notin v^-\\\overline{z_{\alpha}} & \text{if } \alpha\in v^-
\end{cases}.$ Via $\phi_0$, $\phi_1$ becomes an isomorphism of $\underset{i\in Q_0}{\prod}\mathcal{M}(i)$-bimodules.

Induced by $\phi_0$ and $\phi_1$, we get an isomorphism of $\xr$-algebras $$\phi\colon T(Q,\mathcal{M})\longrightarrow T(Q,\mathcal{M'}).$$ Hence we have the following result.

\begin{lem}\label{change mod}
	Keep the setting above. We have an isomorphism of $\xr$-algebras $$\phi\colon T(Q,\mathcal{M})\longrightarrow T(Q,\mathcal{M'})$$ such that $\phi(\mathcal{M}(p))=\mathcal{M'}(p)$ for each path $p$ in $Q$.
\end{lem}

Recall that a \textbf{bound quiver} over a division ring $D$ is a pair $(Q,R)$, where $Q$ is a quiver and $R$ is a subset of the path algebra $DQ$. Let $(Q,\mathcal{M})$ be a v-uniform modulated quiver with $D$ and $(Q,R)$ a bound quiver. We set $$\langle R\rangle_{\mathcal{M}}:=\langle\{\sum_{i=1}^{m}d_i1_{\mathcal{M}(p_i)}\in T(Q,\mathcal{M})\;|\;\sum_{i=1}^{m}d_ip_i\in R\}\rangle,$$ an ideal of $T(Q,\mathcal{M})$ generated by $R$. We remark that in this case, any ideal $I$ of $T(Q,\mathcal{M})$ can be generated by $R=\{\sum_{i=1}^{m}d_ip_i\;|\;\sum_{i=1}^{m}d_i1_{\mathcal{M}(p_i)}\in I\}$. Moreover, for the cases in Lemma \ref{pathalg}, we have $$T(Q,\mathcal{M})/\langle R\rangle_{\mathcal{M}}\simeq DQ/\langle R\rangle.$$

\medskip
\subsection{Construction of the quiver $\Gamma$}

Given a modulated quiver $(Q,\mathcal{M})$, we construct $\Gamma=(\Gamma_0,\Gamma_1,s,t)$ from it. This construction appeared in \cite[Section 9.3]{G} for a general setting. We are inspired by \cite[Section 6]{DD} and describe the construction in more details.

\textbf{Vertices of $\Gamma$}:
\begin{enumerate}
	\item Each $i\in Q_0$ with $\mathcal{M}(i)=\mathbb{R}$ or $\mathbb{H}$ gives a vertex $i$ in $\Gamma_0$.
	\item Each $i\in Q_0$ with $\mathcal{M}(i)=\mathbb{C}$ gives two vertices $i$ and $\overline{i}$ in $\Gamma_0$.
\end{enumerate}

\textbf{Arrows of $\Gamma$}: Let $\alpha:i\rightarrow j\in Q_1$ with $\mathcal{M}(\alpha)=S,$ a simple bimodule.
\begin{enumerate}
	\item If $S\in\mathbb{R}\mbox{-}\mathbb{R}\mbox{-mod}$ or $\mathbb{H}\mbox{-}\mathbb{H}\mbox{-mod}$, it gives an arrow $\alpha:i\rightarrow j$ in $\Gamma_1$.
	\item If $S\in\mathbb{R}\mbox{-}\mathbb{C}\mbox{-mod}$ or $\mathbb{H}\mbox{-}\mathbb{C}\mbox{-mod}$, it gives arrows $\alpha:i\rightarrow j$ and $\overline{\alpha}:\overline{i}\rightarrow j$ in $\Gamma_1$. \item If $S\in\mathbb{C}\mbox{-}\mathbb{R}\mbox{-mod}$ or $\mathbb{C}\mbox{-}\mathbb{H}\mbox{-mod}$, it gives arrows $\alpha:i\rightarrow j$ and $\overline{\alpha}:i\rightarrow\overline{j}$ in $\Gamma_1$. \item If $S\in\mathbb{R}\mbox{-}\mathbb{H}\mbox{-mod}$ or $\mathbb{H}\mbox{-}\mathbb{R}\mbox{-mod}$, it gives arrows $\alpha:i\rightarrow j$ and $\overline{\alpha}:i\rightarrow j$ in $\Gamma_1$.
	\item If $S\simeq\mathbb{C}\in\mathbb{C}\mbox{-}\mathbb{C}\mbox{-mod}$, it gives arrows $\alpha:i\rightarrow j$ and $\overline{\alpha}:\overline{i}\rightarrow\overline{j}$ in $\Gamma_1$.
	\item If $S\simeq\overline{\mathbb{C}}\in\mathbb{C}\mbox{-}\mathbb{C}\mbox{-mod}$, it gives arrows $\alpha:\overline{i}\rightarrow j$ and $\overline{\alpha}:i\rightarrow\overline{j}$ in $\Gamma_1$.
\end{enumerate}

Then the quiver $\Gamma=(\Gamma_0,\Gamma_1,s,t)$ is well-defined and there is an \textbf{automorphism} $\tau$ of $\Gamma$ defined as follows.
\begin{enumerate}
	\item  For each $i$ in $\Gamma_0$, if $\overline{i}$ exists, $\tau(i)=\overline{i}$ and $\tau(\overline{i})=i$; if not, $\tau(i)=i$.
	\item  For each $\alpha$ in $\Gamma_1$, if $\overline{\alpha}$ exists, $\tau(\alpha)=\overline{\alpha}$ and $\tau(\overline{\alpha})=\alpha$; if not, $\tau(\alpha)=\alpha$.
\end{enumerate}

We also have a surjection $$\pi\colon\Gamma\rightarrow Q,$$ where $\pi(i)=i=\pi(\overline{i})$ (if $\overline{i}$ exists), $\forall i\in Q_0$, and $\pi(\alpha)=\alpha=\pi(\overline{\alpha})$ (if $\overline{\alpha}$ exists), $\forall \alpha\in Q_1$. We say that $j$ in $\Gamma_0$ is a \textbf{fiber} of $i$ in $Q_0$ if $\pi(j)=i$ and that $\beta$ in $\Gamma_1$ is a \textbf{fiber} of $\alpha$ in $Q_1$ if $\pi(\beta)=\alpha$. A path $p'$ in $\Gamma$ is called a \textbf{fiber} of a path $p$ in $Q$ if the arrows in $p'$ are fibers of those in $p$ in orders. Clearly, each path in $\Gamma$ is a fiber of some (unique) path in $Q$.

\begin{exm}
{\rm 1)} Let $Q=i\overset{\alpha}{\rightarrow }j\overset{\beta}{\rightarrow} k\overset{\gamma}{\rightarrow }l$, $\mathcal{M}(i)=\xh$, $\mathcal{M}(j)=\xr$, $\mathcal{M}(k)=\xc$ and $\mathcal{M}(l)=\xh$. Then $$\Gamma=\makecell{\xymatrix@R=1ex{ & & k \ar[dr]^{\gamma} & \\i \ar@<.5ex>[r]^{\alpha} \ar@<-.5ex>[r]_{\overline{\alpha}}&j\ar[ur]^{\beta}\ar[dr]_{\overline{\beta}}& &l\\ &  &\overline{k}\ar[ur]_{\overline{\gamma}}& }}.$$ In this example, paths $\gamma\beta\alpha$, $\overline{\gamma}\overline{\beta}\alpha$, $\gamma\beta\overline{\alpha}$ and $\overline{\gamma}\overline{\beta}\overline{\alpha}$ in $\Gamma$ are fibers of $\gamma\beta\alpha$ in $Q$.\\
{\rm 2)} If $\mathcal{M}$ is v-uniform with $\xr$ or $\xh$, then $\Gamma=Q$. \\
{\rm 3)} If $\mathcal{M}$ is v-uniform with $\xc$ and $\mathcal{M}(\alpha)=\xc$ for each arrow $\alpha$, then $\Gamma=Q\sqcup Q$.\\
{\rm 4)} Let $Q=\xymatrix@R=1ex{i\ar@<.5ex>[r]^{\alpha} \ar@<-.5ex>[r]_{\beta}&j}$, $\mathcal{M}(i)=\mathcal{M}(j)=\xc$, $\mathcal{M}(\alpha)=\xc$ and $\mathcal{M}(\beta)=\overline{\xc}$. Then $$\Gamma=\makecell{\xymatrix@R=3ex@C=12ex{j&i\ar[l]_{\alpha}\ar[dr]^{\overline{\beta}}&\\ &\overline{i}\ar[ul]^{\beta}\ar[r]_{\overline{\alpha}}&\overline{j}}}.$$
\end{exm}

\subsection{The quiver presentation of $T(Q,\mathcal{M})/I\otimes_{\xr}\xc$}

Given a modulated quiver $(Q,\mathcal{M})$, we investigate the complexified algebra $T(Q,\mathcal{M})\otimes_{\xr}\xc$. 

There is a canonical isomorphism of algebras $$\psi_1\colon T(Q,\mathcal{M})\otimes_{\xr}\xc\simeq T(\underset{i\in Q_0}{\prod}(\mathcal{M}(i)\otimes_{\xr}\xc),\underset{\alpha\in Q_1}{\bigoplus}(\mathcal{M}(\alpha)\otimes_{\xr}\xc)),$$ mapping $(a_i)_{i\in Q_0}\otimes\lambda$ to $(a_i\otimes\lambda)_{i\in Q_0}$, $(m_{\alpha})_{\alpha\in Q_1}\otimes\lambda$ to $(m_{\alpha}\otimes\lambda)_{\alpha\in Q_1}$, and $(m_n\otimes m_{n-1}\otimes\cdots\otimes m_1)\otimes \lambda$ to $(m_n\otimes \lambda)\otimes(m_{n-1}\otimes 1_{\xc})\otimes \cdots\otimes (m_1\otimes 1_{\xc})$. 

For each $\alpha: i\rightarrow j\in Q_1$, set
\begin{equation*}
	C_{\alpha}=\begin{cases}
		_{\mathcal{M}(j)}\overline{C}_{\mathcal{M}(i)} & \text{if } \mathcal{M}(\alpha)={_{\xc}\overline{\xc}_{\xc}}\\
		_{\mathcal{M}(j)}C_{\mathcal{M}(i)} &\text{otherwise}
	\end{cases}.
\end{equation*} For bimodules $_{\mathcal{M}(j)}C_{\mathcal{M}(i)}$ and $ _{\mathcal{M}(j)}\overline{C}_{\mathcal{M}(i)}$, see Subsections 2.1-2.12. Then the explicit isomorphisms in Subsections 2.1-2.12 induce an isomorphism of tensor algebras $$\psi_2:T(\underset{i\in Q_0}{\prod}(\mathcal{M}(i)\otimes_{\xr}\xc),\underset{\alpha\in Q_1}{\bigoplus}(\mathcal{M}(\alpha)\otimes_{\xr}\xc))\simeq T(\underset{i\in Q_0}{\prod}C_{\mathcal{M}(i)},\underset{\alpha: i\rightarrow j\in Q_1}{\bigoplus}C_{\alpha}).$$

Denote by $\mathbf{e}$ the idempotent of $T(Q,\mathcal{M})\otimes_{\xr}\xc$ such that
$$\psi_2\psi_1(\mathbf{e})=(e_{C_{\mathcal{M}(i)}})_{i\in Q_0}\in\underset{i\in Q_0}{\prod}C_{\mathcal{M}(i)},$$ where
\begin{equation*}
e_{C_{\mathcal{M}(i)}}=\begin{cases}
1_{C_{\mathcal{M}(i)}} & \text{if } \mathcal{M}(i)=\xr \text{ or } \xc\\
[\begin{smallmatrix}
1&0\\0&0\end{smallmatrix}] &\text{if } \mathcal{M}(i)=\xh
\end{cases}.\end{equation*}

By the definition of $\mathbf{e}$, $T(Q,\mathcal{M})\otimes_{\mathbb{R}}\xc$ is Morita equivalent to $\mathbf{e}(T(Q, \mathcal{M})\otimes_{\xr}\xc)\mathbf{e}$. We show that $\Gamma$ is a quiver presentation of $T(Q, \mathcal{M})\otimes_{\xr}\xc$; compare \cite[Section 9.3]{G}.

\begin{prop}\label{morita path alg}
Let $\Gamma$ be the quiver constructed by a modulated quiver $(Q,\mathcal{M})$. Then there is an isomorphism of $\xc$-algebras
$$\Psi\colon \mathbf{e}(T(Q, \mathcal{M})\otimes_{\xr}\xc)\mathbf{e}\longrightarrow\xc \Gamma.$$ 
\end{prop}
\begin{proof}
Denote $\underset{i\in Q_0}{\prod}C_{\mathcal{M}(i)}$ by $D$, $\underset{\alpha\in Q_1}{\bigoplus}C_{\alpha}$ by $M$ and $(e_{C_{\mathcal{M}(i)}})_{i\in Q_0}$ by $\epsilon$. We already have an isomorphism of algebras $$\psi_2\psi_1\colon\mathbf{e}(T(Q,\mathcal{M})\otimes_{\xr}\xc)\mathbf{e}\simeq\epsilon T(D,M)\epsilon.$$

(1). We construct an isomorphism $$\psi_3\colon\epsilon T(D,M)\epsilon\simeq T(\epsilon D\epsilon,\epsilon M\epsilon).$$ It can be obtained once there are isomorphisms of $D$-$D$-bimodules $$\underbrace{M\otimes_{D}M\otimes_{D}\cdots\otimes_{D}M}_{n \mbox{ copies}}\simeq M\epsilon \otimes_{\epsilon D\epsilon }\underbrace{\epsilon M\epsilon \otimes_{\epsilon D\epsilon }\cdots\otimes_{\epsilon D\epsilon} \epsilon M\epsilon}_{n-2\mbox{ copies}}\otimes_{\epsilon D\epsilon }\epsilon M, n\geq 2.$$ 

We construct these isomorphisms inductively on $n$ and note that we only need to consider each direct summand $\psi_2\psi_1(\mathcal{M}(p)\otimes_{\xr}\xc)$, where $p$ is a path with length larger than one in $Q$.

For any path $\beta\alpha:i\rightarrow j\rightarrow k$ with length two in $Q$, we construct an isomorphism of $C_{\mathcal{M}(k)}$-$C_{\mathcal{M}(i)}$-bimodules $$
f_2\colon C_{\beta}\otimes_{C_{\mathcal{M}(j)}}C_{\alpha}\simeq C_{\beta}e_{C_{\mathcal{M}(j)}}\otimes_{(e_{C_{\mathcal{M}(j)}}C_{\mathcal{M}(j)}e_{C_{\mathcal{M}(j)}})}e_{C_{\mathcal{M}(j)}}C_{\alpha}.$$If $\mathcal{M}(j)\neq \xh$, then $e_{C_{\mathcal{M}(j)}}=1_{C_{\mathcal{M}(j)}}$ and we set $f_2$ to be the identity map. If $\mathcal{M}=\xh$, then ${_{\mathcal{M}(k)}C_{\mathcal{M}(j)}}$ and ${_{\mathcal{M}(j)}C_{\mathcal{M}(i)}}$ are $M_2(\xc)$ with the $\mathcal{M}(j)$-module structure given by matrix multiplication, see Subsections 2.7-2.11. In this case, set $f_2$ to be the composition of  $M_2(\xc)\otimes_{M_2(\xc)}M_2(\xc)\simeq M_2(\xc)$ given by matrix multiplication and 
\begin{align*}
M_2(\xc)\simeq&M_2(\xc)\begin{bmatrix}1&0\\0&0\end{bmatrix}\otimes_{[\begin{smallmatrix}1&0\\0&0\end{smallmatrix}]M_2(\xc)[\begin{smallmatrix}1&0\\0&0\end{smallmatrix}]}\begin{bmatrix}1&0\\0&0\end{bmatrix} M_2(\xc)\mbox{ given by}\\ \begin{bmatrix}a&b\\c&d\end{bmatrix}\mapsto&\begin{bmatrix}a&0\\c&0\end{bmatrix}\otimes \begin{bmatrix}1&0\\0&0\end{bmatrix}+\begin{bmatrix}b&0\\d&0\end{bmatrix}\otimes\begin{bmatrix}0&1\\0&0\end{bmatrix}.
\end{align*} Hence we obtain an isomorphism $M\otimes_{D}M\simeq M\epsilon\otimes_{\epsilon D\epsilon }\epsilon M$ and we also denote it by $f_2$.

Assume that for $n=k-1$, an isomorphism $f_{k-1}$ is constructed. Then for $n=k$, the isomorphism we want is $$f_n:=(f_2\otimes\underbrace{\mathrm{Id}_{\epsilon M\epsilon}\otimes\dots\otimes\mathrm{Id}_{\epsilon M\epsilon}}_{k-3\mbox{ copies}}\otimes\mathrm{Id}_{\epsilon M})\circ(\mathrm{Id}_{M}\otimes f_{k-1}).$$

(2). Finally, we construct an isomorphism $$\psi_4\colon T(\epsilon D\epsilon, \epsilon M\epsilon)\simeq \xc\Gamma.$$ So we need to construct an isomorphism of algebras $\epsilon D\epsilon\simeq \prod_{i\in \Gamma_0}\xc e_i$ and an isomorphism of bimodules $\epsilon M\epsilon\simeq \oplus_{\alpha\in \Gamma_1}\xc\alpha$.

For each $i\in Q_0$, we construct isomorphisms of algebras as below.\\
$\cdot$ If $\mathcal{M}(i)=\xr$, we have $e_{C_{\xr}}(C_{\xr})e_{C_{\xr}}\simeq \xc e_i \mbox{, given by } a\mapsto ae_i.$\\
$\cdot$ If $\mathcal{M}(i)=\xh$, we have $e_{C_{\xh}}(C_{\xh})e_{C_{\xh}}\simeq \xc e_i \mbox{, given by } \begin{bmatrix}1&0\\0&0\end{bmatrix}\begin{bmatrix} a & b \\ c & d \end{bmatrix}\begin{bmatrix}1&0\\0&0\end{bmatrix}\mapsto ae_i.$\\
$\cdot$ If $\mathcal{M}(i)=\xc$, we have
$e_{C_{\xc}}(C_{\xc})e_{C_{\xc}}\simeq \xc e_i\times\xc e_{\overline{i}} \mbox{, given by } (a,b)\mapsto (ae_i,be_{\overline{i}}).$

For each $\alpha\in Q_1$, via the isomorphisms of algebras above, we construct below isomorphisms of $\xc e_{t(\alpha)}$-$\xc e_{s(\alpha)}$-bimodules.\\
$\cdot$ If $\mathcal{M}(s(\alpha))=\xr$ and $\mathcal{M}(t(\alpha))=\xr$, we have 
$$e_{C_\xr}(_\xr C_{\xr})e_{C_{\xr}}\simeq \xc\alpha \mbox{, given by } a\mapsto a\alpha.$$
$\cdot$ If $\mathcal{M}(s(\alpha))=\xc$ and $\mathcal{M}(t(\alpha))=\xr$, we have 
$$e_{C_\xr}(_\xr C_{\xc})e_{C_{\xc}}\simeq \xc\alpha\oplus\xc\overline{\alpha} \mbox{, given by } (a,b)\mapsto (a\alpha,b\overline{\alpha}).$$
$\cdot$ If $\mathcal{M}(s(\alpha))=\xr$ and $\mathcal{M}(t(\alpha))=\xc$, we have 
$$e_{C_\xc}(_\xc C_{\xr})e_{C_{\xr}}\simeq \xc\alpha\oplus\xc\overline{\alpha} \mbox{, given by } (a,b)\mapsto (a\alpha,b\overline{\alpha}).$$
$\cdot$ If $\mathcal{M}(s(\alpha))=\xh$ and $\mathcal{M}(t(\alpha))=\xr$, we have 
$$e_{C_\xr}(_\xr C_{\xh})e_{C_{\xh}}\simeq \xc\alpha\oplus\xc\overline{\alpha} \mbox{, given by } \begin{bmatrix} a & b \\ c & d \end{bmatrix}\begin{bmatrix}1&0\\0&0\end{bmatrix}\mapsto (a\alpha,c\overline{\alpha}).$$
$\cdot$ If $\mathcal{M}(s(\alpha))=\xr$ and $\mathcal{M}(t(\alpha))=\xh$, we have 
$$e_{C_\xh}(_\xh C_{\xr})e_{C_{\xr}}\simeq \xc\alpha\oplus\xc\overline{\alpha} \mbox{, given by } \begin{bmatrix}1&0\\0&0\end{bmatrix}\begin{bmatrix} a & b \\ c & d \end{bmatrix}\mapsto (a\alpha,b\overline{\alpha}).$$
$\cdot$ If $\mathcal{M}(s(\alpha))=\xh$ and $\mathcal{M}(t(\alpha))=\xh$, we have 
$$e_{C_\xh}(_\xh C_{\xh})e_{C_{\xh}}\simeq \xc\alpha \mbox{, given by } \begin{bmatrix}1&0\\0&0\end{bmatrix}\begin{bmatrix} a & b \\ c & d \end{bmatrix}\begin{bmatrix}1&0\\0&0\end{bmatrix}\mapsto a\alpha.$$
$\cdot$ If $\mathcal{M}(s(\alpha))=\xc$ and $\mathcal{M}(t(\alpha))=\xh$, we have 
$$e_{C_\xh}(_\xh C_{\xc})e_{C_{\xc}}\simeq \xc\alpha\oplus\xc\overline{\alpha} \mbox{, given by } \begin{bmatrix}1&0\\0&0\end{bmatrix}\begin{bmatrix} a & b \\ c & d \end{bmatrix}\mapsto (a\alpha,b\overline{\alpha}).$$
$\cdot$ If $\mathcal{M}(s(\alpha))=\xh$ and $\mathcal{M}(t(\alpha))=\xc$, we have 
$$e_{C_\xc}(_\xc C_{\xh})e_{C_{\xh}}\simeq \xc\alpha\oplus\xc\overline{\alpha} \mbox{, given by } \begin{bmatrix} a & b \\ c & d \end{bmatrix}\begin{bmatrix}1&0\\0&0\end{bmatrix}\mapsto (a\alpha,c\overline{\alpha}).$$
$\cdot$ If $\mathcal{M}(s(\alpha))=\xc$ and $\mathcal{M}(t(\alpha))=\xc$, we have 
$$e_{C_\xc}(_\xc C_{\xc})e_{C_{\xc}}\simeq \xc\alpha\oplus\xc\overline{\alpha} \mbox{, given by } (a,b)\mapsto (a\alpha,b\overline{\alpha});$$
$$e_{C_\xc}(_\xc \overline{C}_{\xc})e_{C_{\xc}}\simeq \xc\alpha\oplus\xc\overline{\alpha} \mbox{, given by } (a,b)\mapsto (a\alpha,b\overline{\alpha}).$$

Now let $\psi_4$ be the isomorphism induced by the isomorphisms above. Hence $$\Psi=\psi_4\psi_3\psi_2\psi_1\colon\mathbf{e}(T(Q, \mathcal{M})\otimes_{\xr}\xc)\mathbf{e}\longrightarrow\xc \Gamma$$ is an isomorphism we want.
\end{proof}

\begin{rem}\label{cresp-path}
	Let $\Psi$ be the isomorphism constructed in the proof of Proposition \ref{morita path alg}. For each path $p$ in $Q$, $$\Psi( \mathbf{e}(\mathcal{M}(p)\otimes\xc)\mathbf{e})=\oplus_{i=1}^{n}\xc q_i,$$ where $q_1,\dots,q_n$ are all the paths in $\Gamma$ which are fibers of $p$.
\end{rem}

By Gabriel's theorem, each finite-dimensional real algebra $A$ is Morita equivalent to $T(Q,\mathcal{M})/I$ for some modulated quiver $(Q,\mathcal{M})$ and some admissible ideal $I$. We call $T(Q,\mathcal{M})/I$ a \textbf{modulated quiver presentation} of $A$. Each finite-dimensional complex algebra $A$ is Morita equivalent to $\xc Q/I$ for some quiver $Q$ and admissible ideal $I$, and we call $\xc Q/I$ a \textbf{quiver presentation} of $A$; see \cite{B} for details. In this paper, $\mathcal{M}(\alpha)$ is always a simple bimodule for each arrow $\alpha$. Such treatment is convenient for us and makes no essential difference from the usual setting in \cite{B}.

For each admissible ideal $I$ of $T(Q,\mathcal{M})$, set $$J:=\Psi(\mathbf{e}(I\otimes_{\mathbb{R}}\xc)\mathbf{e})\subset\xc \Gamma.$$ Then $J$ is an admissible ideal of $\xc \Gamma$ by the above remark. Hence we have 

\begin{thm}\label{morita eq}
The complexified algebra $T(Q, \mathcal{M})/I\otimes_{\mathbb{R}}\xc$ is Morita equivalent to $\xc \Gamma/J$. 
\end{thm}

\begin{rem}
	The quiver $\Gamma$ is independent with the choice of idempotents and isomorphisms. From now on, we always fix the idempotent $\mathbf{e}$, the isomorphism $\Psi$, and then the admissible ideal $J$. We call $\xc\Gamma/J$ the \textbf{complexified quiver presentation} of $T(Q,\mathcal{M})/I$.
\end{rem}

The admissible ideal $J$ is not easy to describe without computation. Let's see an example below.  
\begin{exm}\label{aft-cpl-npth}
	Let $$Q=\xymatrix@R=2ex{ i\ar[r]^{\alpha}& j \ar@(ul,ur)^{\beta}\ar[r]^{\gamma} & k}, \mbox{ and } \mathcal{M}(i)=\mathcal{M}(j)=\mathcal{M}(k)=\xh.$$ Then $\mathcal{M}(\alpha)=\mathcal{M}(\beta)=\mathcal{M}(\gamma)=\xh$ and $\Gamma=Q$. 
	
	Take $1_{\mathcal{M}(\alpha)}\in\mathcal{M}(\alpha)$, $1_{\mathcal{M}(\beta)}\in\mathcal{M}(\beta)$, $1_{\mathcal{M}(\gamma)}\in\mathcal{M}(\gamma)$, $\lambda\in\xh$ and set $$I=\langle 1_{\mathcal{M}(\gamma)}\otimes 1_{\mathcal{M}(\alpha)}+\lambda(1_{\mathcal{M}(\gamma)}\otimes 1_{\mathcal{M}(\beta)}\otimes 1_{\mathcal{M}(\alpha)}), 1_{\mathcal{M}(\beta)}\otimes 1_{\mathcal{M}(\beta)}\rangle\subset T(Q,\mathcal{M}).$$
	
	{\rm 1)}. When $\lambda\in\xr=\{[\begin{smallmatrix}a&0\\0&a
	\end{smallmatrix}]\,|\,a\in\xr\}\subset\xh$, $$I=\{h(1_{\mathcal{M}(\gamma)}\otimes 1_{\mathcal{M}(\alpha)}+\lambda(1_{\mathcal{M}(\gamma)}\otimes 1_{\mathcal{M}(\beta)}\otimes 1_{\mathcal{M}(\alpha)}))\,|\,h\in\xh\}\oplus\left\langle \mathcal{M}(\beta\beta)\right\rangle.$$ Applying the isomorphism $\Psi$ in Proposition \ref{morita path alg}, we have \begin{align*}
	&\Psi(\mathbf{e}((1_{\mathcal{M}(\gamma)}\otimes 1_{\mathcal{M}(\beta)}\otimes 1_{\mathcal{M}(\alpha)})\otimes 1_{\xc})\mathbf{e})\\=
	&\psi_4\psi_3(\begin{bmatrix}1&0\\0&0\end{bmatrix}\psi_2((1_{\mathcal{M}(\gamma)}\otimes 1_{\xc})\otimes(1_{\mathcal{M}(\beta)}\otimes 1_{\xc})\otimes(1_{\mathcal{M}(\alpha)}\otimes 1_{\xc}))\begin{bmatrix}1&0\\0&0\end{bmatrix})\\=
	&\psi_4\psi_3( \begin{bmatrix}1&0\\0&0\end{bmatrix}(\begin{bmatrix}1&0\\0&1\end{bmatrix}\otimes\begin{bmatrix}1&0\\0&1\end{bmatrix}\otimes\begin{bmatrix}1&0\\0&1\end{bmatrix})\begin{bmatrix}1&0\\0&0\end{bmatrix})\\ =&\psi_4(\begin{bmatrix}1&0\\0&0\end{bmatrix}(\begin{bmatrix}1&0\\0&0\end{bmatrix}\otimes\begin{bmatrix}1&0\\0&0\end{bmatrix}\otimes\begin{bmatrix}1&0\\0&0\end{bmatrix})\begin{bmatrix}1&0\\0&0\end{bmatrix}), \mbox{since }[\begin{smallmatrix}0&1\\0&0\end{smallmatrix}][\begin{smallmatrix}1&0\\0&0\end{smallmatrix}]=0
	\\=&\gamma\beta\alpha,
	\end{align*} 
	and $\Psi(\mathbf{e}((1_{\mathcal{M}(\gamma)}\otimes 1_{\mathcal{M}(\alpha)})\otimes 1)\mathbf{e})=\gamma\alpha$.
	Hence $$\Psi(\mathbf{e}((1_{\mathcal{M}(\gamma)}\otimes 1_{\mathcal{M}(\alpha)}+\lambda(1_{\mathcal{M}(\gamma)}\otimes 1_{\mathcal{M}(\beta)}\otimes 1_{\mathcal{M}(\alpha)}))\otimes 1)\mathbf{e})=\gamma\alpha+\lambda(\gamma\beta\alpha).$$ By Remark \ref{cresp-path}, $\Psi( \mathbf{e}(\mathcal{M}(\beta\beta)\otimes\xc)\mathbf{e})=\xc\beta\beta.$ Therefore $J=\langle \gamma\alpha+\lambda\gamma\beta\alpha,\beta\beta\rangle.$
	
	{\rm 2)}. When $\lambda\notin\xr$, there is a $\mu$ in $\xh$ such that $\lambda\mu\neq\mu\lambda$. Therefore \begin{align*}
	&1_{\mathcal{M}(\gamma)}\otimes 1_{\mathcal{M}(\beta)}\otimes 1_{\mathcal{M}(\alpha)}\\=&(\lambda\mu+\mu\lambda)^{-1}((1_{\mathcal{M}(\gamma)}\otimes 1_{\mathcal{M}(\alpha)}+\lambda(1_{\mathcal{M}(\gamma)}\otimes 1_{\mathcal{M}(\beta)}\otimes 1_{\mathcal{M}(\alpha)}))\mu\\&-\mu(1_{\mathcal{M}(\gamma)}\otimes 1_{\mathcal{M}(\alpha)}+\lambda(1_{\mathcal{M}(\gamma)}\otimes 1_{\mathcal{M}(\beta)}\otimes 1_{\mathcal{M}(\alpha)})))\in I.
	\end{align*} Then \begin{align*}
	I&=\langle 1_{\mathcal{M}(\gamma)}\otimes 1_{\mathcal{M}(\alpha)}, 1_{\mathcal{M}(\gamma)}\otimes 1_{\mathcal{M}(\beta)}\otimes 1_{\mathcal{M}(\alpha)}, 1_{\mathcal{M}(\beta)}\otimes 1_{\mathcal{M}(\beta)}\rangle\\&=\mathcal{M}(\gamma\alpha)\oplus\mathcal{M}(\gamma\beta\alpha)\oplus\left\langle \mathcal{M}(\beta\beta)\right\rangle.
	\end{align*} By Remark \ref{cresp-path},
	$J=\langle \gamma\alpha,\gamma\beta\alpha,\beta\beta\rangle.$
\end{exm}

We can describe $J$ in following cases.

\begin{lem}\label{rtc}
	Let $\xc\Gamma/J$ be the complexified quiver presentation of $T(Q,\mathcal{M})/I$ and $p$ be a path of $Q$.\\
	{\rm 1)}. If $\mathcal{M}$ is v-uniform with $\xr$ or $\xh$, then $p$ in $\Gamma$ is the only fiber of $p$ in $Q$. Moreover, $\mathcal{M}(p)\subset I$ if and only if $p\in J$. \\
	{\rm 2)}. If $\mathcal{M}$ is v-uniform with $\xc$, then there are exactly two paths $p'$ and $p''$ in $\Gamma$ being fibers of $p$ in $Q$. Moreover, $\mathcal{M}(p)\subset I$ if and only if $p'\in J$ if and only if $p''\in J$.
\end{lem}
\begin{proof}
	1) The first statement follows from the construction of $\Gamma$. By Remark \ref{cresp-path}, we have $\mathcal{M}(p)\subset I\Leftrightarrow \xc p\subset J\Leftrightarrow p\in J$. 
	
	2) The first statement follows from the construction of $\Gamma$. By Remark \ref{cresp-path}, we have $\mathcal{M}(p)\subset I\Leftrightarrow\xc p'\oplus\xc p''\subset J\Leftrightarrow p',p''\in J.$ On the other hand, we assume that $p'\in J$ or $p''\in J$. By Remark \ref{cresp-path}, there is a nonzero element in $\mathcal{M}(p)$ belonging to $I$. Since $\mathcal{M}(p)$ is a simple bimodule generated by any nonzero element, $\mathcal{M}(p)\subset I$.
\end{proof}

\section{classification of derived-discrete real algebras}
 
\subsection{Gentle one-cycle without clock condition}
Since the tensor algebra of a modulated quiver is generally not a path algebra, there is usually no definition for gentle algebras in these cases. To our purpose, we only need to generalize the definition slightly. We use notions in Section 3.1.

\begin{defn}
	Let $(Q,R)$ be a bound quiver. A vertex $v\in Q_0$ is called \textbf{gentle} if it satisfies the following conditions.\\
	1). The number of arrows starting at $v$ is at most two and so is the number of arrows terminating at $v$.\\
	2). For any $\alpha, \beta, \gamma \in Q_1$ such that $\beta\neq\gamma$ and $t(\beta)=t(\gamma)=s(\alpha)=v$, exactly one of $\alpha\beta$ and $\alpha\gamma$ belongs to $R$.\\
	3). For any $\alpha, \beta, \gamma \in Q_1$ such that $\beta\neq\gamma$ and $s(\beta)=s(\gamma)=t(\alpha)=v$, exactly one of $\beta\alpha$ and $\gamma\alpha$ belongs to $R$.
\end{defn}

We recall some definitions from \cite{V}. A bound quiver $(Q,R)$ is called \textbf{gentle}, if $R$ is a set consisting of paths with length two and all the vertices of $Q$ are gentle in $(Q,R)$. A connected quiver $Q$ is called \textbf{one-cycle} if it contains exactly one (unoriented) cycle. Further, a gentle one-cycle bound quiver $(Q,R)$ is called \textbf{without clock condition} if in the cycle, the number of clockwise-oriented paths in $R$ is different from that of counterclockwise-oriented paths in $R$.

\begin{defn}
	A connected $\xr$-algebra is called \textbf{gentle one-cycle without clock condition}, if it admits a modulated quiver presentation $T(Q,\mathcal{M})/I$ such that \\
	1). the modulation $\mathcal{M}$ is v-uniform,\\
	2). there is a gentle one-cycle without clock condition bound quiver $(Q, R)$ such that $I=\langle R\rangle_{\mathcal{M}}$ (see notations after Lemma \ref{pathalg}).
\end{defn}

\begin{rem}
	1) A connected $\xc$-algebra is called gentle one-cycle without clock condition if it admits a quiver presentation $\xc\Gamma/\langle J\rangle$ such that the bound quiver $(\Gamma,J)$ is gentle one-cycle without clock condition, see \cite{V}.\\
	2) By Lemma \ref{pathalg}, sometimes $T(Q,\mathcal{M})$ is isomorphic to a path algebra for a modulated quiver $(Q,\mathcal{M})$. In these cases, we assume that $D$ is $\xr$, $\xh$ or $\xc$ and $(Q,R)$ is a gentle one-cycle without clock condition bound quiver. Then the $\xr$-algebra $DQ/\left\langle R\right\rangle \simeq\xr Q/\left\langle R\right\rangle\otimes_{\xr}D$ is gentle one-cycle without clock condition.\\
	3) Not all the gentle one-cycle without clock condition real algebras are isomorphic to some quotients of path algebras; see the example below.
\end{rem}

\begin{exm}\label{notpathalg}
	Let $$Q=\makecell{\xymatrix{i\ar@(ul,ur)^{\alpha}}}\mbox{ , }\mathcal{M}(i)=\xc\mbox{, }\mathcal{M}(\alpha)=\overline{\xc}\mbox{, }I=\left\langle\mathcal{M}(\alpha\alpha)\right\rangle.$$ Then $T(Q,\mathcal{M})/I$ is gentle one-cycle without clock condition. If $T(Q,\mathcal{M})/I\simeq \xc Q'/I'$ for some $Q'$ and $I'$, then the complexified quivers $\Gamma$ and $\Gamma'$ of $Q$ and $Q'$ respectively should be the same. But $\Gamma={\xymatrix{i\ar@/^/[r]^{\alpha}& \overline{i}\ar@/^/[l]^{\overline{\alpha}}}}$ while $\Gamma'=Q'\sqcup Q'$.
\end{exm}

\subsection{The classification of real derived-discrete algebras}
Now we can state our main theorem and leave the proof to Sections 5 and 6.
\begin{thm}\label{main}
	Let $\xc\Gamma/J$ be the complexified quiver presentation of $T(Q, \mathcal{M})/I$. Then $T(Q, \mathcal{M})/I$ is gentle one-cycle without clock condition if and only if each connected component of $\xc\Gamma/J$ is gentle one-cycle without clock condition.
\end{thm}

Recall that a $k$-algebra $A$ is called \textbf{derived-discrete} if given any cohomology dimension vector over $k$, $A$ admits only finitely many objects in its bounded derived category up to isomorphisms. By \cite{V}, a connected $\xc$-algebra is derived-discrete if and only if it is either piecewise hereditary of Dynkin type or gentle one-cycle without clock condition. We prove the classification theorem for derived-discrete $\xr$-algebras.

\begin{thm}\label{cls}
	A connected $\xr$-algebra is derived-discrete if and only if it is either piecewise hereditary of Dynkin type or gentle one-cycle without clock condition.
\end{thm}
\begin{proof}
	By \cite{L}, an $\xr$-algebra $A$ is derived-discrete if and only if so is $A\otimes_{\xr}\xc$. We claim that $A$ is piecewise hereditary of Dynkin type if and only if so is $A\otimes_{\xr}\xc$. Then we can use the classification of derived-discrete $\xc$-algebras and Theorem \ref{main} to prove our conclusion. 
	
	Proof of the claim. Recall that an algebra $A$ is called derived finite if up to shifts and isomorphisms there are only finitely many indecomposable objects in the bounded derived category $\dba$. By \cite[Theorem 2.3]{Z}, an algebra is piecewise hereditary of Dynkin type if and only if it is derived finite. 
	So we only need to prove that $A$ is derived finite if and only if so is $A\otimes_{\xr}\xc$. 
	
	We use an argument for split and separable algebra extensions; see \cite{L}. For the `if' part. If $A\otimes_{\xr}\xc$ is derived finite, let $\{M_1,,M_2,\dots,M_m\}$ be all the indecomposable objects up to shifts and isomorphisms in $\mathbf{D}^b(A\otimes_{\xr}\xc\mbox{-mod})$. Each $M_i$ can be viewed as an object in $\dba$. For each indecomposable object $X$ in $\dba$, $X$ is a direct summand of $X\otimes_{\xr}\xc$ in $\dba$. Thus there is an $M_i$ such that up to shifts and isomorphisms, $M_i$ is a direct summand of $X\otimes_{\xr}\xc$ in $\mathbf{D}^b(A\otimes_{\xr}\xc\mbox{-mod})$ and $X$ is a direct summand of $M_i$ in $\mathbf{D}^b(A\mbox{-mod})$. Therefore, up to shifts and isomorphisms, all the indecomposable direct summands of $M_i$, $i\in\{1,\dots,m\}$, in $\dba$ contains all the indecomposable objects, which means that $A$ is derived finite. The `only if' part can be proved with a similar argument. 
\end{proof}

\section{The `only if' part}

\begin{lem}\label{vuniformtogentle}
	Let $(Q,\mathcal{M})$ be a v-uniform modulated quiver and $\xc\Gamma/J$ the complexified quiver presentation of $T(Q,\mathcal{M})/I$, where $I=\langle R\rangle_{\mathcal{M}}$. Then a vertex $v$ is gentle in $(Q,R)$ if and only if each fiber of $v$ is gentle in $(\Gamma,J)$.
\end{lem}
\begin{proof}
	We only need to consider the arrows connected with $v$. Using the construction of $\Gamma$ and Lemma \ref{rtc}, the gentle conditions can be checked case by case. 
\end{proof}

Recall that a \textbf{chain} in a quiver $Q$ is an unoriented path in the underlying graph of $Q$ such that each arrow of $Q$ occurs at most once. We usually denote a chain with length $n-1$ by $v_1-v_2-\cdots -v_n$, where $v_1,\dots, v_n$ are vertices in $Q$ or by $c_1c_2\cdots c_{n-1}$, where $c_1,\dots, c_{n-1}$ are arrows or inverse arrows in $Q$. We say that a chain $c'_1c'_2\cdots c'_m$ in $\Gamma$ is a fiber of a chain $c_1c_2\cdots c_m$ in $Q$ if for each $i$, $c'_i$ is a fiber of $c_i$ and being arrows or inverse arrows simultaneously.  If two graphs $G_1$ and $G_2$ are adjacent (connected by a chain with length one), we denote them by $G_1-G_2$. The lemma below follows immediately from the construction of $\Gamma$.

\begin{lem}\label{fiber of chain}
	Let $(Q,\mathcal{M})$ be a modulated quiver and $c$ a chain in $Q$ from vertex $u$ to vertex $w$. \\
	{\rm 1)}. If there is a vertex $v$ in $c$ with $\mathcal{M}(v)=\xr$ or $\xh$, then in $\Gamma$, each fiber of $c$ contains $v$ and there is a fiber of $c$ from $u$ to $w$.\\
	{\rm 2)}. If $\mathcal{M}(v)=\xc$ for each vertex $v$ in $c$ and the number of arrows with modulation $\overline{\xc}$ in $c$ is even, then in $\Gamma$, there are exactly two fibers of $c$. One is from $u$ to $w$ and the other is from $\overline{u}$ to $\overline{w}$. \\
	{\rm 3)}. If $\mathcal{M}(v)=\xc$ for each vertex $v$ in $c$ and the number of arrows with modulation $\overline{\xc}$ in $c$ is odd, then in $\Gamma$, there are exactly two fibers of $c$. One is from $\overline{u}$ to $w$ and the other is from $u$ to $\overline{w}$.
\end{lem} 

Recall from Lemma \ref{change mod} that if $(Q,\mathcal{M})$ is a modulated quiver such that $\mathcal{M}$ is v-uniform, then for certain vertices, we can change the modulation on arrows connected to them. We can do more in our situation.
\begin{lem}\label{oneornon}
	Let $(Q,\mathcal{M})$ be a modulated quiver such that $Q$ is one-cycle with the cycle $Q^{\circ}$ and $\mathcal{M}$ is v-uniform with $\xc$. Then we have a modulated quiver $(Q,\mathcal{M}')$ and an isomorphism of algebras $T(Q,\mathcal{M})\simeq T(Q,\mathcal{M}')$ satisfying that:\\
	{\rm 1)}. If the number of arrows in $\{\alpha\in Q^{\circ}_1\;|\;\mathcal{M}(\alpha)=\overline{\xc}\}$ is even, then  $$\mathcal{M'}(\alpha)=\xc, \forall\alpha\in Q_1.$$
	{\rm 2)}. If the number of arrows in $\{\alpha\in Q^{\circ}_1\;|\;\mathcal{M}(\alpha)=\overline{\xc}\}$ is odd, then  $$\mathcal{M'}(\alpha)=\begin{cases}
	\overline{\xc}& \mbox{if } \alpha=\beta\\\xc &\mbox{if }\alpha\neq\beta
	\end{cases}, \forall \alpha\in Q_1, \mbox{ for some arrow } \beta \mbox{ in } Q^\circ.$$

\end{lem}
\begin{proof}
    We first modify the modulation on the cycle. If $Q^\circ$ is a loop, then we set $\mathcal{M'}=\mathcal{M}$ and take the identity map as an isomorphism. Now assume that the length of $Q^\circ$ is larger than one. If there is a chain $v_1-v_2-\cdots-v_n$, $n\geq 3$, in $Q^\circ$ such that $\mathcal{M}$ maps arrows to $\xc$ except the first and last ones, then we apply Lemma~\ref{change mod} on $v_2,\dots,v_{n-1}$ so that we get a new modulation which maps all arrows in the chain to $\xc$. This process reduces the number of arrows in $Q^\circ$ with modulation $\overline{\xc}$. Repeat this process till there is no such a chain. Then we obtain a modulation such that the number of arrows on $Q^\circ$ with modulation $\overline{\xc}$ is one or zero.

    For the modulation outside the cycle. If there is an arrow $\alpha$ not in $Q^\circ$ such that $\mathcal{M}(\alpha)=\overline{\xc}$, then we can apply Lemma~\ref{change mod} on the endpoint of $\alpha$ which is farther to $Q^\circ$. So we get a new modulation which maps $\alpha$ to $\xc$. Since $Q$ is one-cycle, we can get $\mathcal{M'}$ by repeating this process on arrows from near to far away from $Q^\circ$.
    
    The isomorphism of algebras is the composition of isomorphisms in Lemma \ref{change mod} according to our modification processes.
\end{proof}
\begin{rem}
	Let $\phi$ be the isomorphism in the above lemma. By Lemma \ref{change mod}, $\phi(\mathcal{M}(p))=\mathcal{M'}(p)$ for each path $p$ in $Q$. Hence $T(Q,\mathcal{M})/I$ is gentle one-cycle without clock condition if and only if so is $T(Q,\mathcal{M}')/\phi(I)$.
\end{rem}

\begin{proof}[\textbf{Proof of the `only if' part of Theorem \ref{main}}]
	
	If $\mathcal{M}$ is v-uniform with $\xr$ or $\xh$, then $\Gamma=Q$ and $\xc\Gamma/J$ is gentle one-cycle without clock condition by Lemma \ref{rtc} and Lemma \ref{vuniformtogentle}. 
	
	If $\mathcal{M}$ is v-uniform with $\xc$, by Lemma \ref{oneornon}, we only need to consider two cases: there is no arrow or only one arrow in the cycle with modulation $\overline{\xc}$. By Lemma~\ref{vuniformtogentle}, we only need to prove that each component of $\xc\Gamma/J$ is one cycle and without clock condition. 
	
	In the first case, $\Gamma=Q\sqcup Q$ and each connected component of $\xc\Gamma/J$ is one-cycle without clock condition by Lemma \ref{rtc}. In the second case, by Lemma \ref{fiber of chain}, two fibers of the cycle in $Q$ form the only cycle in $\Gamma$. To see that it is without clock condition, we use the following facts. The fibers of paths having uniform orientation in the cycle of $Q$ also have uniform orientation in the cycle of $\Gamma$ (by construction of $\Gamma$) and each path in $I$ has exactly two fibers in $J$ (by Lemma \ref{fiber of chain}). Hence the difference of numbers of clockwise and counterclockwise relations in the cycle of $\Gamma$ doubles that of $Q$. 
\end{proof}

\section{The `if' part}
\subsection{The possible forms of $\xc\Gamma/J$}The algebra $\xc\Gamma/J$ is gentle one-cycle without clock condition doesn't mean that $J$ is monomial (i.e. can be generated by paths). To prove the theorem, we describe all the possible forms of $J$. 

Denote by $P_i$ (\textit{resp.} $P'_i$) and $S_i$ (\textit{resp.} $S'_i$) the corresponding projective and simple $\xc\Gamma/J$ (\textit{resp.} $\xc\Gamma/J'$) modules for each $i\in\Gamma_0$. An algebra isomorphism $\xc \Gamma/J\simeq\xc \Gamma/J'$ induces an isomorphism between module categories and projectives. Hence it induces uniquely a permutation $\sigma$ on $\Gamma_0$ such that for any $i,j$ in $\Gamma_0$, $$\mathrm{Hom}_{\xc\Gamma/J'}(P'_i,P'_j)\simeq\mathrm{Hom}_{\xc\Gamma/J}(P_{\sigma(i)},P_{\sigma(j)}),$$ $$\mathrm{Ext}^1_{\xc\Gamma/J'}(S'_i,S'_j)\simeq\mathrm{Ext}^1_{\xc\Gamma/J}(S_{\sigma(i)},S_{\sigma(j)}).$$ The last isomorphism implies that there is an automorphism on $\Gamma_1$ compatible with $\sigma$. Hence $\sigma$ induces a quiver automorphism on $\Gamma$ and then a $\xc$-algebra automorphism on $\xc\Gamma$, which are also denoted by $\sigma$.

\begin{prop}\label{special}
	Assume that we have an algebra isomorphism $\xc \Gamma/J\simeq\xc \Gamma/J'$ such that $J'=\langle p_1,p_2,\dots,p_r\rangle$ and $(\Gamma,\{p_1,\dots,p_r\})$ is gentle one-cycle without clock condition, and that the automorphism on $\Gamma_0$ induced by the isomorphism as above is the identity map. Then the following statements hold.\\
	{\rm a)}. If $J$ is monomial, then $J=\langle p_1,\dots,p_r\rangle$. \\
	{\rm b)}. If $J$ is not monomial, then $\Gamma$ contains the following quiver 
	$$\Omega:\makecell{\xymatrix@R=2ex{v_1 \ar[r]^{\alpha_1} & v_2 \ar@{.}[r] & v_{n-1} \ar[dl]^{\alpha_{n-1}}\\
		u \ar[r]_\beta& v_0 \ar[ul]^{\alpha_0}\ar[r]_\gamma& w
	}}$$ as a subquiver, where $n\geq 1$, $\alpha_{n-1}\cdots\alpha_1\alpha_0\notin J'$. Moreover, we have $\gamma\beta\in J'$ {\rm (}set $\gamma\beta=p_1${\rm )} and there is a nonzero $\lambda\in\xc$ such that $J=\langle \gamma\beta-\lambda\gamma\alpha_{n-1}\cdots\alpha_0\beta,p_2,\dots,p_r\rangle$.
\end{prop}
    To prove this proposition, we give some observations first.

    By assumption, we have $$\mathrm{Hom}_{\xc\Gamma/J'}(P'_i,P'_j)\simeq\mathrm{Hom}_{\xc\Gamma/J}(P_i,P_j),\forall i,j\in \Gamma_0.$$

    For each path $i\overset{\beta}{\longrightarrow}j\overset{\gamma}{\longrightarrow}k$ in $\Gamma$ such that $\gamma\beta\in J'$, we prove that $\gamma\beta\in J$ in the following cases. Assumptions of these cases are strengthened progressively until the situation of $\Gamma$ in b) occurs. Denote the cycle in $\Gamma$ by $\Gamma^\circ$.
    
    \medskip
	(\textbf{\uppercase\expandafter{\romannumeral1}}). If $i=k$, then $\gamma\beta\in J$.\\
	1). If $i=k=j$, then $\beta: i\rightarrow j(=i)$ is the cycle of $\Gamma$. Since $\Gamma$ is one-cycle, there is no other path except $\beta^n$, $n\geq 1$, from $i$ to $i$. Since $\gamma\beta=\beta^2\in J'$, we have $\dim_{\xc}\mathrm{End}_{\xc\Gamma/J}(P_i)=\dim_{\xc}\mathrm{End}_{\xc\Gamma/J'}(P'_i)=2.$ Therefore $\gamma\beta\in J$. \\	
	2). If $i=k\neq j$, then $i\rightarrow j\rightarrow i$ is the cycle of $\Gamma$. We have $\gamma\beta\in J$ by using a similar argument as above. 
	
	\medskip	
	If $i\neq k$, then $i\neq j$ and $j\neq k$, otherwise $(\Gamma,\{p_1,\dots,p_r\})$ contains a cycle without relations on it. Hence we assume that $i,j,k$ are pairwise different in the following cases. Given two sets of morphisms $M$ and $N$ in a module category, set $M\circ N:=\{f\circ g\,|\,f\in M,g\in N\}.$ To prove that $\gamma\beta\in J$, it is sufficient to prove that $$\mathrm{Hom}_{\xc\Gamma/J'}(P'_j,P'_i)\circ\mathrm{Hom}_{\xc\Gamma/J'}(P'_k,P'_j)=0\mbox{ or }\mathrm{Hom}_{\xc\Gamma/J'}(P'_k,P'_i)=0.$$ Denote by $\Gamma^\circ\cap\gamma\beta$ a subset of $\{i,j,k\}$ consisting of vertices which belong to $\Gamma^\circ$.
	
	(\textbf{\uppercase\expandafter{\romannumeral2}}). If $i,j,k$ are pairwise different and $\Gamma^\circ\cap\gamma\beta\neq\{j\}$, then $\gamma\beta\in J$.	\\
	1). If $\Gamma^\circ\cap\gamma\beta=\emptyset$, then there is no other path from $i$ to $k$ since $\Gamma$ is one-cycle. Therefore $\mathrm{Hom}_{\xc\Gamma/J'}(P'_k,P'_i)=0$. \\	
	2). If $\Gamma^\circ\cap\gamma\beta=\{i\}$ or $\{k\}$, then $\beta$ and $\gamma$ are not in $\Gamma^\circ$. Each path from $i$ to $k$ contains $\gamma\beta$. Therefore $\mathrm{Hom}_{\xc\Gamma/J'}(P'_k,P'_i)=0.$\\
	3). If $\Gamma^\circ\cap\gamma\beta=\{i,j\}$ (or $\{j,k\}$), then $\beta$ (or $\gamma$) is in $\Gamma^\circ$ since $\Gamma$ is one-cycle. Because $(\Gamma,\{p_1,\dots,p_r\})$ is without clock condition, other paths from $i$ to $j$ (or $j$ to $k$) must belong to $J'$. Therefore $\mathrm{Hom}_{\xc\Gamma/J'}(P'_j,P'_i)\circ\mathrm{Hom}_{\xc\Gamma/J'}(P'_k,P'_j)=0.$\\
	4). If $\Gamma^\circ\cap\gamma\beta\supseteq\{i,k\}$ then  $\Gamma^\circ\cap\gamma\beta=\{i,j,k\}$ and $\gamma\beta$ is in $\Gamma^\circ$ since $\Gamma$ is one-cycle. Each path passing $i$, $j$ and $k$ contains $\gamma\beta$. Therefore $\mathrm{Hom}_{\xc\Gamma/J'}(P'_j,P'_i)\circ\mathrm{Hom}_{\xc\Gamma/J'}(P'_k,P'_j)=0$.
	
	\medskip
	(\textbf{\uppercase\expandafter{\romannumeral3}}). If $i,j,k$ are pairwise different, $\Gamma^\circ\cap\gamma\beta=\{j\}$ and $\Gamma^\circ$ is not an oriented cycle, then $\gamma\beta\in J$. Indeed, there is no other path from $i$ to $k$. Therefore $\mathrm{Hom}_{\xc\Gamma/J'}(P'_k,P'_i)=0$.
	
	\medskip
	Finally, assume that $i,j,k$ are pairwise different, $\Gamma^\circ\cap\gamma\beta=\{j\}$ and $\Gamma^\circ$ is an oriented cycle. Then $\Gamma$ contains $\Omega$ as a subquiver and $i=u$, $j=v_0$, $k=w$. 
	
	(\textbf{\uppercase\expandafter{\romannumeral4}}). If further $\alpha_{n-1}\cdots\alpha_0\in J'$, then $\mathrm{Hom}_{\xc\Gamma/J'}(P'_w,P'_u)=0$. Thus $\gamma\beta\in J$.
	
	\medskip	
	In summary, if $\gamma\beta$ is in one of the above cases, then $\gamma\beta\in J$. Now we prove the proposition.

\begin{proof}
	If $\gamma\beta$ is not in cases                     (\uppercase\expandafter{\romannumeral1}), (\uppercase\expandafter{\romannumeral2}), (\uppercase\expandafter{\romannumeral3}) or (\uppercase\expandafter{\romannumeral4}), the assumption:
	\begin{center}
		$\gamma\beta:u\rightarrow v_0\rightarrow w\subseteq\Omega\subseteq\Gamma$ and $\alpha_{n-1}\cdots\alpha_0\notin J'$ (\textreferencemark),
	\end{center} must hold. In this case, $\alpha_0\alpha_{n-1}\in J'$ since $(\Gamma,\{p_1,\dots,p_r\})$ is gentle and $\gamma\beta\in J'$. By case (\uppercase\expandafter{\romannumeral1}) or case  (\uppercase\expandafter{\romannumeral2}) 4), $\alpha_0\alpha_{n-1}\in J$. Therefore $$\dim_{\xc}\mathrm{Hom}_{\xc\Gamma/J}(P_w,P_u)=\dim_{\xc}\mathrm{Hom}_{\xc\Gamma/J'}(P'_w,P'_u)=1,$$ where $\mathrm{Hom}_{\xc\Gamma/J}(P_w,P_u)$ can be spanned by $\gamma\beta+J$ and $\gamma\alpha_{n-1}\cdots\alpha_0\beta+J$ over $\xc$.

	a). Assume that $J$ is monomial. By cases (\uppercase\expandafter{\romannumeral1}), (\uppercase\expandafter{\romannumeral2}), (\uppercase\expandafter{\romannumeral3}) and (\uppercase\expandafter{\romannumeral4}), each path $p_i\in J'$ not satisfying the assumption (\textreferencemark) belongs to $J$. If $p_i(=\gamma\beta)\in J'$ satisfies the assumption (\textreferencemark), exactly one of $\gamma\beta$ and $\gamma\alpha_{n-1}\cdots\alpha_0\beta$ belongs to $J$ because $J$ is monomial. Since $(\Gamma,\{p_1,\dots,p_r\})$ is gentle, $\gamma\alpha_{n-1},\alpha_0\beta\notin J'$. Hence $\gamma\alpha_{n-1}\cdots\alpha_0\beta\notin J'$. So \begin{align*}
	{}&\mathrm{Hom}_{\xc\Gamma/J}(P_{w},P_{v_0})\circ\mathrm{rad}(\mathrm{End}_{\xc\Gamma/J}(P_{v_0}))\circ\mathrm{Hom}_{\xc\Gamma/J}(P_{v_0},P_{u})\\
	\simeq&\mathrm{Hom}_{\xc\Gamma/J'}(P'_w,P'_{v_0})\circ\mathrm{rad}(\mathrm{End}_{\xc\Gamma/J'}(P'_{v_0}))\circ\mathrm{Hom}_{\xc\Gamma/J'}(P'_{v_0},P'_u)\neq0.
	\end{align*} Therefore $\gamma\alpha_{n-1}\cdots\alpha_0\beta\notin J$ and $\gamma\beta\in J$. 
	
	Hence $\langle p_1,\dots,p_r\rangle\subseteq J$. By comparing the $\xc$-dimension of algebras, we have $$J=\langle p_1,\dots,p_r\rangle.$$
	
	b). Assume that $J$ is not monomial. Then there is one $p_i$ satisfying the assumption (\textreferencemark), otherwise each $p_i$ belongs to $J$ by (\uppercase\expandafter{\romannumeral1}), (\uppercase\expandafter{\romannumeral2}), (\uppercase\expandafter{\romannumeral3}) and (\uppercase\expandafter{\romannumeral4}) and $J$ is monomial. Set $\gamma\beta=p_1$ satisfying the assumption (\textreferencemark) without loss of generality. Then there is a nonzero $\lambda\in\xc$ such that $\gamma\beta-\lambda(\gamma\alpha_{n-1}\cdots\alpha_0\beta)\in J$. Other $p_i$ will not satisfy (\textreferencemark) since  $(\Gamma,\{p_1,\dots,p_r\})$ is gentle one-cycle and $\alpha_0\alpha_{n-1}\in J'$. Hence $$J'':=\langle \gamma\beta-\lambda(\gamma\alpha_{n-1}\cdots\alpha_0\beta),p_2,\dots,p_r\rangle\subseteq J.$$ There is an isomorphism of algebras $\theta\colon\xc \Gamma/J'\rightarrow\xc \Gamma/J''$ induced by
	$$\forall i\in\Gamma_0, \theta(e_i)=e_i;$$
	$$\forall\alpha\in\Gamma_1, \theta(\alpha)=\begin{cases}
	\alpha & \text{if } \alpha\neq\gamma\\
	\gamma-\lambda\gamma\alpha_{n-1}\cdots\alpha_0 & \text{if } \alpha=\gamma
	\end{cases}.$$
	So $J=J''=\langle \gamma\beta-\lambda(\gamma\alpha_{n-1}\cdots\alpha_0\beta),p_2,\dots,p_r\rangle.$
\end{proof}
\begin{rem}\label{gentle point}
	In the above proposition, if b) holds, then all the vertices except $v_0$ in $(\Gamma,\{\gamma\beta-\lambda(\gamma\alpha_{n-1}\cdots\alpha_0\beta),p_2,\dots,p_r\})$ are gentle.
\end{rem}

\subsection{Necessary conditions on $(Q,\mathcal{M})$} To prove the 'if' part of the main theorem, we should prove that $\mathcal{M}$ is v-uniform. For this purpose, following lemmas are needed.

\begin{lem}\label{tr-or-cy}
	Let $(Q,\mathcal{M})$ be a connected modulated quiver and $\xc\Gamma$ be the complexified quiver presentation of $T(Q, \mathcal{M})$. Assume that each connected component of $\Gamma$ is one-cycle. Then the following statements hold.\\
	{\rm 1)}. The quiver $Q$ is a tree or one-cycle. \\
	{\rm 2)}. If $Q$ is a tree, then $Q$ contains either a chain $u-v$ with $\mathcal{M}(u)=\xr$ and $\mathcal{M}(v)=\xh$ or a chain $v_1-v_2-\cdots-v_n$, $n>2$ such that $\mathcal{M}(v_1)\neq\xc$, $\mathcal{M}(v_i)=\xc$ for $i\in\{2,3\dots,n-1\}$, and $\mathcal{M}(v_n)\neq\xc$. 
\end{lem}
\begin{proof}
	1) Given a cycle in $Q$, by Lemma \ref{fiber of chain} the subgraph formed by its fibers in $\Gamma$ contains a cycle. Since different cycles in $Q$ has different fibers in $\Gamma$, our statement holds
	
	2) If the modulation of $T(Q,\mathcal{M})$ is v-uniform and $Q$ is a tree, then $\Gamma$ is also a tree. Hence there is a chain $u-v$ in $Q$ with $\mathcal{M}(u)\neq\mathcal{M}(v)$. If $\mathcal{M}(u)=\xr$ and $\mathcal{M}(v)=\xh$, then nothing needs to be proved. Now we assume that there is no such chain. Thus there are vertices with modulation $\xc$ in $Q$.
	
	Let $T_0$ be a maximal subtree of $Q$ with uniform modulation $\xr$ or $\xh$ on vertices, and $T_1,T_2,\dots,T_m$ be maximal subtrees of $Q$ which are adjacent to $T_0$ and with uniform modulation $\xc$ on vertices. There are vertices not belonging to any $T_i$, $i=0,\dots,m$, otherwise $\Gamma$ is a tree. Among these vertices, by assumption above there is a vertex $v$ not adjacent to $T_0$ but adjacent to $T_i$ for some $i\in\{1,\dots,m\}$. Therefore $\mathcal{M}(v)\neq\xc$ and there is a chain we want in the subquiver $v-T_i-T_0$.
\end{proof}

Let $(Q,\mathcal{M})$ be a modulated quiver such that $Q$ a tree and $\xc\Gamma/J$ be the complexified quiver presentation of $T(Q, \mathcal{M})/I$. For each path $\alpha\beta\colon u\rightarrow v\rightarrow w$ in $Q$ with length two such that $\mathcal{M}(v)=\xc$, there are exactly two paths $p$ and $\tau(p)$ in $\Gamma$ which are fibers of $\alpha\beta$.
\begin{lem}\label{sym}
	Given the setting above, we have that $p\in J$ if and only if $\tau(p)\in J$.
\end{lem}
\begin{proof}
	If $\mathcal{M}(u)\neq\mathcal{M}(w)$ or $\mathcal{M}(u)=\mathcal{M}(w)=\xc$, then $\mathcal{M}(\alpha\beta)$ is a simple bimodule which can be generated by any nonzero element. By Remark \ref{cresp-path}, $$p\in J\Leftrightarrow \mathcal{M}(\alpha\beta)\subset I\Leftrightarrow \tau(p)\in J.$$
	
	So we can assume that $\mathcal{M}(u)=\mathcal{M}(w)=\xr$ or $\xh$. In both cases, $$\Gamma=\makecell{\xymatrix@R=1ex{ & v \ar[dr]^{\alpha} & \\u\ar[ur]^{\beta}\ar[dr]_{\overline{\beta}}& &w\\ &\overline{v}\ar[ur]_{\overline{\alpha}}& }}.$$ Set $p=\alpha\beta, \tau(p)=\overline{\alpha}\overline{\beta}.$
	Since $\tau(\tau(p))=p$, we only need to prove that $$p\in J\Rightarrow\tau(p)\in J.$$
	
	1). We assume that $\mathcal{M}(u)=\mathcal{M}(w)=\xr$. For any $a\in\mathcal{M}(\alpha)=\xc$, $b\in\mathcal{M}(\beta)=\xc$, we have \begin{align*}
	&\Psi(\mathbf{e}((a\otimes b)\otimes 1_{\xc})\mathbf{e})
	=\psi_4\psi_3(1_{\mathcal{M}(\xr)}\psi_2((a\otimes 1_{\xc})\otimes(b\otimes 1_{\xc}))1_{\mathcal{M}(\xr)})\\
	=&\psi_4\psi_3((a,\overline{a})\otimes(b,\overline{b}))
	=\psi_4((a,\overline{a})\otimes(b,\overline{b}))
	=(a\alpha,\overline{a}\overline{\alpha})\otimes(b\beta,\overline{b}\overline{\beta})=abp+\overline{a}\overline{b}\tau(p).
	\end{align*} If $p\in J$, we can find $a_i\in\mathcal{M}(\alpha)$, $b_i\in\mathcal{M}(\beta)$ and $\lambda_i\in\xc^*$ such that $a_i\otimes b_i\in I$ for $i\in\{1,2,\dots,s\}$ and $$\Psi(\mathbf{e}(\sum_{i=1}^{s}(a_i\otimes b_i)\otimes\lambda_i)\mathbf{e})=\sum_{i=1}^{s}\lambda_ia_ib_ip+\lambda_i\overline{a_i}\overline{b_i}\tau(p)=p.$$ This implies that $\sum_{i=1}^{s}\lambda_ia_ib_i=1$ and $\sum_{i=1}^{s}\lambda_i\overline{a_i}\overline{b_i}=0$. Therefore $$\Psi(\mathbf{e}(\sum_{i=1}^{s}(a_i\otimes b_i)\otimes\overline{\lambda_i})\mathbf{e})=\sum_{i=1}^{s}\overline{\lambda_i}a_ib_ip+\overline{\lambda_i}\overline{a_i}\overline{b_i}\tau(p)=\tau(p).$$ Since $a_i\otimes b_i\in I$ for $i\in\{1,2,\dots,s\}$, we have $\tau(p)\in J$. 
	
	2) We assume that $\mathcal{M}(u)=\mathcal{M}(w)=\xh$. For any $\begin{bmatrix}a\\b\end{bmatrix}\in\mathcal{M}(\alpha)={_\xh\xc^2_\xc}$ and $\begin{bmatrix}c&d\end{bmatrix}\in\mathcal{M}(\beta)={_\xc\xc^2_\xh}$, we have
	\begin{align*}
	&\Psi(\mathbf{e}((\begin{bmatrix}a\\b\end{bmatrix}\otimes\begin{bmatrix}c&d\end{bmatrix})\otimes1_{\xc})\mathbf{e})\\
	=&\psi_4\psi_3(\begin{bmatrix}1&0\\0&0\end{bmatrix}\psi_2((\begin{bmatrix}a\\b\end{bmatrix}\otimes1_{\xc})\otimes(\begin{bmatrix}c&d\end{bmatrix}\otimes1_{\xc}))\begin{bmatrix}1&0\\0&0\end{bmatrix})\\
	=&\psi_4\psi_3(\begin{bmatrix}1&0\\0&0\end{bmatrix}\begin{bmatrix}a&-\overline{b}\\b&\overline{a}\end{bmatrix}\otimes\begin{bmatrix}c&d\\-\overline{d}&\overline{c}\end{bmatrix}\begin{bmatrix}1&0\\0&0\end{bmatrix})\\
	=&\psi_4(\begin{bmatrix}1&0\\0&0\end{bmatrix}\begin{bmatrix}a&-\overline{b}\\b&\overline{a}\end{bmatrix}\otimes\begin{bmatrix}c&d\\-\overline{d}&\overline{c}\end{bmatrix}\begin{bmatrix}1&0\\0&0\end{bmatrix})\\
	=&(a\alpha, -\overline{b}\overline{\alpha})\circ(c\beta,-\overline{d}\beta)=acp+\overline{b}\overline{d}\tau(p).
	\end{align*}  If $p\in J$, we can find $\begin{bmatrix}a_i\\b_i\end{bmatrix}\in\mathcal{M}(\alpha)$, $\begin{bmatrix}c_i&d_i\end{bmatrix}\in\mathcal{M}(\beta)$ and $\lambda_i\in\xc^*$ such that $\begin{bmatrix}a_i\\b_i\end{bmatrix}\otimes \begin{bmatrix}c_i&d_i\end{bmatrix}\in I$ for $i\in\{1,2,\dots,s\}$ and $$\Psi(\mathbf{e}(\sum_{i=1}^{s}(\begin{bmatrix}a_i\\b_i
	\end{bmatrix}\otimes \begin{bmatrix}c_i&d_i\end{bmatrix})\otimes\lambda_i)\mathbf{e})=\sum_{i=1}^{s}\lambda_ia_ic_ip+\lambda_i\overline{b_i}\overline{d_i}\tau(p)=p.$$ This implies that $\sum_{i=1}^{s}\lambda_ia_ic_i=1$ and $\sum_{i=1}^{s}\lambda_i\overline{b_i}\overline{d_i}=0$. Therefore 
	\begin{align*}
	&\Psi(\mathbf{e}(\sum_{i=1}^{s}(\begin{bmatrix}0&1\\-1&0\end{bmatrix}\begin{bmatrix}a_i\\b_i\end{bmatrix}\otimes\begin{bmatrix}c_i&d_i\end{bmatrix}\begin{bmatrix}0&1\\-1&0\end{bmatrix})\otimes\overline{\lambda_i})\mathbf{e})\\
	=&\Psi(\mathbf{e}(\sum_{i=1}^{s}(\begin{bmatrix}b_i\\-a_i\end{bmatrix}\otimes\begin{bmatrix}-d_i&c_i\end{bmatrix})\otimes\overline{\lambda_i})\mathbf{e})\\ =&\sum_{i=1}^{s}-\overline{\lambda_i}b_id_ip-\overline{\lambda_i}\overline{a_i}\overline{c_i}\tau(p)=-\tau(p).
	\end{align*} 
	Since $\begin{bmatrix}a_i\\b_i\end{bmatrix}\otimes\begin{bmatrix}c_i&d_i\end{bmatrix}\in I$ for $i\in\{1,2,\dots,s\}$ and $\begin{bmatrix}
	0&1\\-1&0
	\end{bmatrix}\in\mathcal{M}(w)= \xh$, we have $\tau(p)\in J$.
\end{proof}
\begin{lem}\label{motri}
	Given an $\xr$-algebra $T(Q, \mathcal{M})/I$, let $\xc\Gamma/J$ be its complexified quiver presentation and $u-v-w$ be a chain in $Q$. If $\mathcal{M}(u)=\mathcal{M}(v)=\xr$ or $\xh$ and the fiber of $v$ is gentle in $(\Gamma,J)$, then $\mathcal{M}(w)=\mathcal{M}(v)=\mathcal{M}(u)$.
\end{lem}
\begin{proof}
	1) If $u-v-w$ is a path $p$ with $\mathcal{M}(w)\neq\mathcal{M}(v)$, one can check case by case that there are exactly two paths in $\Gamma$ which are fibers of $p$ and the bimodule $\mathcal{M}(p)$ can be generated by any non-zero element. By Remark \ref{cresp-path}, fibers of $p$ must simultaneously belong or not belong to $J$. It is ridiculous since $v$ in $\Gamma$ is gentle.\\
	2) If $u-v-w$ is not a path with $\mathcal{M}(w)\neq\mathcal{M}(v)$, then $v$ in $\Gamma$ is a starting or terminal vertex for more than two arrows. This contradicts with that $v$ is gentle.
\end{proof}
\begin{lem}\label{tri mod}
	Let $(Q,\mathcal{M})$ be a connected modulated quiver which is one-cycle and $\xc\Gamma/J$ be the complexified quiver presentation of $T(Q,\mathcal{M})/I$. If each connected component of $\Gamma$ is one-cycle and vertices outside the cycles of $\Gamma$ are gentle in $(\Gamma,J)$, then $\mathcal{M}$ is v-uniform on $Q$.
\end{lem}
\begin{proof}
	Let $c$ be a chain in $Q$ forming $Q^\circ$. If there is a chain $u-v$ in $c$ with $\mathcal{M}(u)\neq\mathcal{M}(v)$, then we can assume that $\mathcal{M}(u)=\xr$ or $\xh$. By Lemma \ref{fiber of chain}, each fiber of $c$ is a cycle containing $u$ in $\Gamma$. Since $u-v$ has exactly two fibers which are in different cycles in $\Gamma$, we get two cycles in one connected component. Hence $\mathcal{M}$ is v-uniform on $Q^\circ$.
	
	If $\mathcal{M}(i)=\xr$ (or $\xh$) for any $i\in Q^{\circ}_0$, then by Lemma \ref{motri} $\mathcal{M}$ is v-uniform on $Q$.
	
	If $\mathcal{M}(i)=\xc$ for $i\in Q^{\circ}_0$, we prove that $\mathcal{M}$ is v-uniform. If not, there is a chain $d\colon v_0-v_1-\cdots-v_t$ in $Q$ such that $v_0,\dots,v_{t-1}\notin Q^{\circ}_0$, $v_t\in Q^{\circ}_0$, $\mathcal{M}(v_1)=\mathcal{M}(v_2)=\cdots=\mathcal{M}(v_t)=\xc$ and $\mathcal{M}(v_0)\neq\xc$. There are exactly two fibers of $d$ in $\Gamma$: one is from $v_0$ to $v_t$ and the other one is from $v_0$ to $\overline{v_t}$. When the number of arrows in $Q^{\circ}$ with modulation $\overline{\xc}$ is even, by Lemma \ref{fiber of chain}, the fibers of $c$ are two cycles in $\Gamma$. These two cycles are connected with fibers of $d$ and thus are in one connected component of $\Gamma$, which contradicts with $\Gamma$ is one cycle. When the number of arrows in $Q^{\circ}$ with modulation $\overline{\xc}$ is odd, by Lemma \ref{fiber of chain} each fiber of $c$, together with fibers of $d$, form a cycle in $\Gamma$. Then there are two cycles in one connected component of $\Gamma$, which is also a contradiction.
\end{proof}

\subsection{Proof of the `if' part of Theorem \ref{main}}
\begin{proof}
	Assume that $Q$ is connected and each connected component of $\xc\Gamma/J$ is gentle one-cycle without clock condition. Hence $\xc\Gamma/J\simeq\xc\Gamma/J'$ for some $J'=\langle p_1,\dots,p_r\rangle$ such that each connected component of $(\Gamma,\{p_1,\dots,p_r\})$ is gentle one-cycle without clock condition. As in Section 6.1, there is an automorphism $\sigma$ on $\mathrm{Aut}(\xc\Gamma)$. Then each connected component of $(\Gamma,\sigma(\{p_1,\dots,p_r\}))$ is also gentle one-cycle without clock condition and $\sigma$ induces an isomorphism of algebras $\xc \Gamma/J'\simeq\xc\Gamma/\sigma(J')$. Replace $\xc\Gamma/J'$ by $\xc\Gamma/\sigma(J')$ and $\xc \Gamma/J\simeq\xc \Gamma/J'$ by the composition $\xc \Gamma/J\simeq\xc \Gamma/J'\simeq\xc\Gamma/\sigma(J')$. Then this new isomorphism induces uniquely an automorphism on $\Gamma_0$ which is the identity map. This allows us to apply Proposition~\ref{special} on each connected component of $\xc\Gamma/J$. By Lemma \ref{tr-or-cy} $Q$ is a tree or one-cycle. 
	
	1) Suppose that $J$ is not monomial. By Proposition \ref{special}, there is a connected component of $\Gamma$ containing $\Omega$ as a subquiver. So $Q$ cannot be a tree otherwise the cycle in this component is not oriented. Therefore $Q$ is one-cycle. By Remark \ref{gentle point} and Lemma \ref{tri mod}, $\mathcal{M}$ is v-uniform. 
	
	If $\mathcal{M}$ is v-uniform with $\xc$ and only one arrow in $Q^{\circ}$ with modulation $\overline{\xc}$, then the cycle in  $\Gamma$ ($\Gamma$ is connected in this case) is formed by two fibers of $Q^{\circ}$. It cannot contain $\Omega$. In fact, $\Omega$ has only one relation $\alpha_{n-1}\alpha_0$ on the cycle, but by Lemma~\ref{rtc} relations on the cycle appear in pairs. Hence if $\mathcal{M}$ is v-uniform with $\xc$, then by Lemma \ref{oneornon} we can assume that $\mathcal{M}(\alpha)=\xc$ for each $\alpha\in Q_1$. Therefore we can assume that $T(Q,\mathcal{M})$ is a path algebra by Lemma \ref{pathalg}.
	
	If $T(Q,\mathcal{M})=\xh Q$, by Proposition \ref{special} $J=\langle\gamma\beta-\lambda\gamma\alpha_{n-1}\cdots\alpha_0\beta,p_2,\dots,p_r\rangle$, where $p_i,\alpha_i,\gamma,\beta$ are paths in $\Gamma$ (or in $Q$), $\lambda\in\xc^*$. By Lemma \ref{rtc} and isomorphisms in Lemma \ref{pathalg}, we have $p_i\in I$, $i=2,\dots,r$, and $\gamma\beta,\gamma\alpha_{n-1}\cdots\alpha_0\beta\notin I$. So $\dim_{\xh}\mathrm{Hom}_{\xh Q/I}(P_w,P_u)=1$ and there is a nonzero $\mu$ in $\xr$ (see Example \ref{aft-cpl-npth}) such that $\gamma\beta-\mu\gamma\alpha_{n-1}\cdots\alpha_0\beta\in I$.  Thus $I=\langle\gamma\beta-\mu\gamma\alpha_{n-1}\cdots\alpha_0\beta,p_2,\dots,p_r\rangle$. Let $I'=\langle\gamma\beta,p_2,\dots,p_r\rangle$. Then $\xh Q/I'$ is gentle one-cycle without clock condition. The map $\theta$ given below:
	$$\forall i\in Q_0, \theta(e_i)=e_i\mbox{; }\forall\alpha\in Q_1, \theta(\alpha)=\begin{cases}
	\alpha & \text{if } \alpha\neq\gamma\\
	\gamma-\mu\gamma\alpha_{n-1}\cdots\alpha_0 & \text{if } \alpha=\gamma
	\end{cases},$$ induces an isomorphism of $\xr$-algebras $\xh Q/I'\rightarrow \xh Q/I$. Consequently, $T(Q,\mathcal{M})/I$ is gentle one-cycle without clock condition. 
	
	The cases of $\xr$ and $\xc$ can be proved by using a similar argument.
	
	2) Suppose that $J$ is monomial. We claim that $Q$ is not a tree. If not, by Lemma~\ref{tr-or-cy} $Q$ contains either a chain $u-v$ with $\mathcal{M}(u)=\xr$ and $\mathcal{M}(v)=\xh$ or a chain $v_1-v_2-\cdots-v_n$, $n>2$ such that $\mathcal{M}(v_1)\neq\xc$, $\mathcal{M}(v_i)=\xc$ for $i\in\{2,3\dots,n-1\}$, and $\mathcal{M}(v_n)\neq\xc$. In each case, the cycle in $\Gamma$ ($\Gamma$ is connected in these cases) is formed by the two fibers of this chain. For the first case, the cycle in $\Gamma$ has no relations on it. For the second one, by Lemma \ref{sym} the relations in the cycle appear in pairs. Thus both cases cannot appear since $\xc\Gamma/J$ is without clock condition by Proposition \ref{special}. 
	
	Finally, suppose that $J$ is monomial and $Q$ is one-cycle. Then $\mathcal{M}$ is v-uniform by Lemma \ref{tri mod}. Therefore $T(Q,\mathcal{M})/I$ is gentle one-cycle without clock condition by Lemma \ref{rtc} and Lemma \ref{vuniformtogentle}.	
\end{proof}

\noindent {\bf Acknowledgements.}\quad The author thanks the referee for the very helpful remarks. He is grateful to Xiao-Wu Chen for his guidance and encouragement during his study and visit in USTC. He also thanks Fei Xu and Chao Zhang for their advices.

This work is supported by the National Natural Science Foundation of China (No.12171297).

\noindent {\bf Data availability.} Data sharing is not applicable to this article since no new data is created or analyzed in this study.

\vskip 10pt

{\footnotesize \noindent Jie Li\\
	School of Mathematics, Hefei University of Technology, Hefei 230000, Anhui, PR China}


\begin{thebibliography}{999}
	\bibitem{B} {\sc D. Benson}, {\em Representations and cohomology}, Cambridge university press {\bf 2}(1) (1991).
	
	\bibitem{BGS} {\sc G. Bobinski, C. Geiss, and A. Skowronski}, {\em Classification of derived discrete algebras}, Cent. Euro. J. Math. {\bf 2}(1) (2004), 19--49.
	
	\bibitem{BPP} {\sc N. Broomhead, D. Pauksztello, and D. Ploog}, {\em Discrete derived categories I: homomorphisms, autoequivalences and t-structures}, Math. Z. {\bf 285}(1-2) (2017), 39--89.
	
	\bibitem{DD} {\sc B. Deng, and J. Du}, {\em Frobenius morphisms and representations of algebras},  Trans. Amer. Math. Soc {\bf 358} (2006), 3591--3622.
	
	\bibitem{DK} {\sc Y. A. Drozd, and V. V. Kirichenko}, {\em Finite dimensional algebras}, Springer Science \& Business Media (2012).
	
	\bibitem{DR} {\sc V. Dlab, and C M. Ringel}, {\em Indecomposable representations of graphs and algebras}, American Mathematical Soc., (1976).
	
	\bibitem{G} {\sc  P. Gabriel}, {\em Indecomposable representations II}, Symposia Mat. Inst. Naz. Alta Mat. {\bf 11} (1973), 81--104.
	
	\bibitem{L} {\sc J. Li}, {\em Algebra extensions and derived-discrete algebras}, arXiv: 1904.07168.
	
	\bibitem{V} {\sc D. Vossieck}, {\em The algebras with discrete derived category}, J. Algebra {\bf 243} (2001), 168--176.

	
	
	\bibitem{Z} {\sc C. Zhang}, {\em Derived Brauer-Thrall type theorem I for artin algebras}, Comm. Algebra {\bf 44} (2016), 3509--3517.
	
\end{thebibliography}
\end{document}